\documentclass[review]{siamart1116}

\usepackage{lipsum}
\usepackage{amsfonts}
\usepackage{graphicx}
\usepackage{epstopdf}
\usepackage{algorithmic}
\usepackage{amsmath}
\usepackage{bm}
\usepackage{graphicx}
\usepackage{caption}
\usepackage{mathrsfs}
\usepackage{subfigure}

\newtheorem{remark}{Remark}
\newtheorem{problem}{Problem}

\ifpdf
  \DeclareGraphicsExtensions{.eps,.pdf,.png,.jpg}
\else
  \DeclareGraphicsExtensions{.eps}
\fi

\numberwithin{theorem}{section}

\newcommand{\TheTitle}{A One-Field Energy-conserving Monolithic Fictitious Domain Method for Fluid-Structure Interactions} 
\newcommand{\TheAuthors}{Yongxing Wang, Peter K. Jimack, and Mark A. Walkley}
\newcommand{\AbrTitle}{A One-Field Energy-conserving Monolithic FDM for FSI}

\headers{\AbrTitle}{\TheAuthors}

\title{{\TheTitle}\thanks{Submitted to the editors on May 10th, 2017.}}

\author{
  Yongxing Wang\thanks{School of Computing, University of Leeds, Leeds, UK, LS2 9JT
(\email{jungsirwang@gmail.com}, \email{scywa@leeds.ac.uk}).}
  \and
  Peter K. Jimack\thanks{School of Computing, University of Leeds, Leeds, UK, LS2 9JT (\email{P.K.Jimack@leeds.ac.uk},
    \email{M.A.Walkley@leeds.ac.uk}).}
  \and
  Mark A. Walkley\footnotemark[3]
}

\usepackage{amsopn}

\ifpdf
\hypersetup{
  pdftitle={\TheTitle},
  pdfauthor={\TheAuthors}
}
\fi

\begin{document}
\nolinenumbers
\maketitle

\begin{abstract}
In this article, we analyze and numerically assess a new fictitious domain method for fluid-structure interactions in two and three dimensions. The distinguishing feature of the proposed method is that it only solves for one velocity field for the whole fluid-structure domain; the interactions remain decoupled until solving the final linear algebraic equations. To achieve this the finite element procedures are carried out separately on two different meshes for the fluid and solid respectively, and the assembly of the final linear system brings the fluid and solid parts together via an isoparametric interpolation matrix between the two meshes. In this article, an implicit version of this approach is introduced. The property of energy conservation is proved, which is a strong indication of stability. The solvability and error estimate for the corresponding stationary problem (one time step of the transient problem) are analyzed. Finally, 2D and 3D numerical examples are presented to validate the conservation properties.
\end{abstract}



\section{Introduction}
\label{sec:introduction}
Three major questions arise when considering a finite element method for the problems of Fluid-Structure Interactions (FSI): (1) what kind of meshes are used (interface fitted or unfitted); (2) how to couple the fluid-structure interactions (monolithic/fully-coupled or partitioned/segregated); (3) what variables are solved (velocity and/or displacement). Combinations of the answers of these questions lead to different types of numerical methods. For examples, \cite{Degroote_2009, K_ttler_2008}  solve for fluid velocity and solid displacement sequentially (partitioned/segregated) using an Arbitrary Lagrangian-Eulerian (ALE) fitted mesh. Whereas \cite{Heil_2004, Heil_2008, Muddle_2012} use an ALE fitted mesh to solve for fluid velocity and solid displacement simultaneously (monolithic/fully-coupled) with a Lagrange Multiplier to enforce the continuity of velocity/displacement on the interface. The Immersed Finite Element Method (IFEM) \cite{Boffi_2015, peskin2002immersed, Wang_2011, Wang_2009,Wang_2013,Zhang_2007,zhang2004immersed} and the Fictitious Domain Method (FDM) \cite{baaijens2001fictitious,Boffi_2016,Glowinski_2001,Hesch_2014,Kadapa_2016,Yu_2005} use two meshes to represent the fluid and solid separately. Although IFEM could be monolithic \cite{Boffi_2015}, the classical IFEM only solves for velocity, while the solid information is arranged on the right-hand side of the fluid equation as a known force term. Although the FDM may be partitioned \cite{Yu_2005}, usually the FDM approach solves for both velocity in the whole domain (fluid plus solid) and displacement of the solid simultaneously via a distributed Lagrange multiplier (DLM) to enforce the consistency of velocity/displacement in the overlapped solid domain.

In the case of one-field and monolithic numerical methods for FSI problems, \cite{Auricchio_2014} introduces a 1D model using a one-field FD formulation based on two meshes, and \cite{Hecht_2017,pironneau2016energy} introduces an energy stable monolithic method (in 2D) based on one Eulerian mesh and discrete remeshing. In a previous study \cite{Wang_2017}, we present a one-field monolithic fictitious domain method which is straightforward to implement in both 2D and 3D. The main features of this method are: (1) only one velocity field is solved in the whole domain, based upon the use of an appropriate $L^2$ projection; (2) the fluid and solid equations are solved monolithically. In this paper, we construct an implicit version of this one-field fictitious domain and, developing the ideas from \cite{Auricchio_2014,pironneau2016energy}, we analyze properties of energy stability and error estimate of the proposed approach.

The paper is organized as follows. Control equations are introduced in \cref{sec:Control equations}. Sections \ref{sec:wfotcl} and \ref{sec:ecotcl} discuss the weak form and the energy estimate in the continuous case respectively. Sections \ref{sec:wfadit} and \ref{sec:ecadit} discuss the weak form and the energy estimate after time discretization respectively. In \cref{sec:aotstp}, the corresponding stationary problem (one time step of the transient problem) is analyzed after time discretization and further after space discretization. In \cref{sec:implementation}, implementation details are presented. Numerical examples are given in \cref{sec:numerical_exs}, and conclusions are presented in \cref{sec:conclusions}.

\section{Control equations}
\label{sec:Control equations}

In the following context, $\Omega_t^f\subset\mathbb{R}^d$ and $\Omega_t^s\subset\mathbb{R}^d$ with $d=2,3$ denote the fluid and solid domain respectively which are time dependent regions as shown in \cref{fig1:Schematic diagram of FSI}. $\Omega=\Omega_t^f \cup \Omega_t^s $ is a fixed domain (with outer boundary $\Gamma$) and $\Gamma_t=\partial\Omega_t^f\cap\partial\Omega_t^s$ is the moving interface between fluid and solid. We denote by ${\bf X}$ the reference (material) coordinates of the solid, by ${\bf x}={\bf x}(\cdot,t)$ the current coordinates of the solid, and by ${\bf x}_0$ is the initial coordinates of the solid. We assume ${{\bf x}(t):\Omega_{\bf X}^s\to\Omega_t^s}$ is one-to-one and invertible with Lipschitz inverse, i.e.: for all $t\in\left[0, T\right]$ and ${\bf s}_1, {\bf s}_2\in\Omega_{\bf X}^s$, $\exists$ $c_0>0$ such that $\|{{\bf x}({\bf s}_1,t)}-{{\bf x}({\bf s}_2,t)}\|_2\ge c_0\|{\bf s}_1-{\bf s}_2\|_2$.

\begin{figure}[h!]
	\centering
	\includegraphics[width=3in,angle=0]{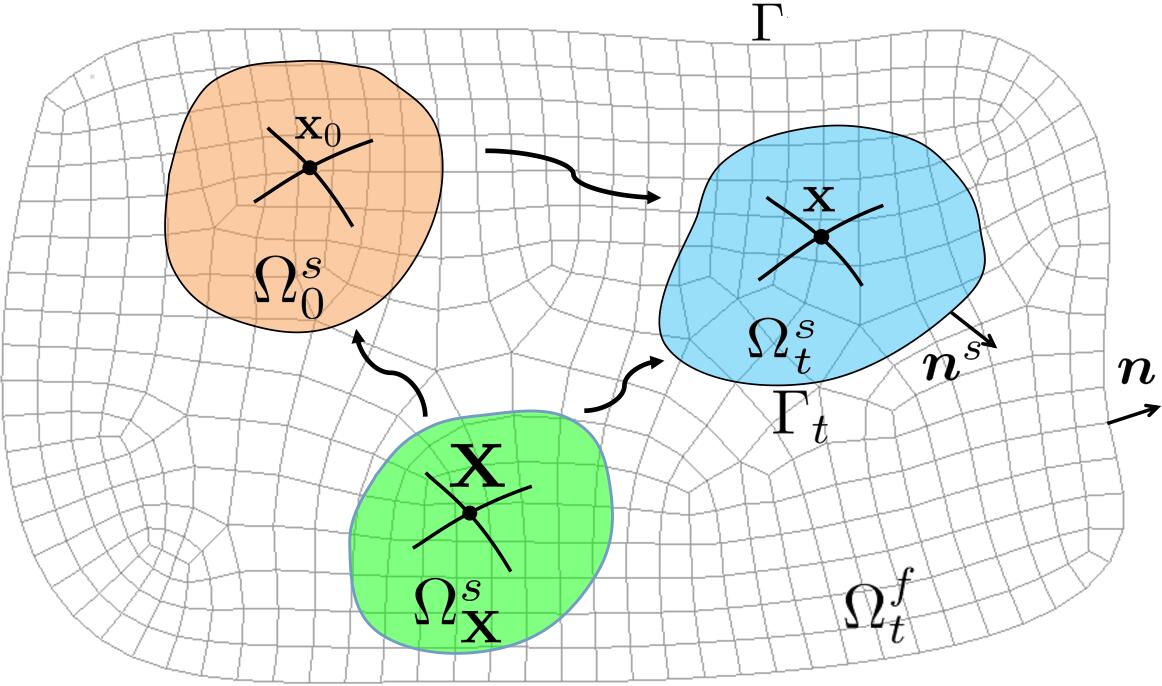}
	\caption {\scriptsize Schematic diagram of FSI, $\Omega=\Omega_t^f\cup \Omega_t^s$.} 
	\label{fig1:Schematic diagram of FSI}
\end{figure}

Let $\rho, {\bf u}, {\bm{\sigma}}$ denote the density, velocity and stress tensor respectively. We assume both an incompressible fluid and incompressible solid, then the conservation of momentum and conservation of mass take the same form as follows:

Momentum equation:
\begin{equation} \label{momentum_equation}
\rho\frac{d{\bf u}}{dt}
=\nabla \cdot {\bm\sigma},
\end{equation}
Continuity equation:
\begin{equation} \label{continuity_equation}
\nabla \cdot {\bf u}=0.
\end{equation}

Let superscripts $f$ and $s$ refer to the fluid and solid respectively, and ${\rm D}{\bf u}=\nabla {\bf u}+\nabla^{\rm T} {\bf u}$, then the constitutive equations may be expressed as follows:

Incompressible Newtonian fluid in $\Omega^f$:
\begin{equation} \label{constitutive_fluid}
{\bm\sigma}={\bm\sigma}^f=\nu^f {\rm D}{\bf u}^f-p^f{\bf I},
\end{equation}
Incompressible viscous-hyperelastic solid in $\Omega^s$ \cite{Boffi_2016}:
\begin{equation} \label{constitutive_solid}
{\bm\sigma}={\bm\sigma}^s={\bm\tau}^s+\nu^s {\rm D}{\bf u}^s-p^s{\bf I},
\end{equation}
where
\begin{equation}\label{expression_of_taus}
{\bm\tau}^s
=J^{-1}{\bf P}{\bf F}^{\rm T},
\end{equation}
with ${\bf P}=\frac{\partial\Psi\left({\bf F}\right)}{\partial{\bf F}}$ (${\bf P}_{ij}=\frac{\partial\Psi\left({\bf F}\right)}{\partial{ F}_{ij}}$) being the first Piola-Kirchhoff stress tensor and $\Psi\left({\bf F}\right)$ being the energy function for a hyperelastic solid material. $J=det_{\bf F}$ is the determinant of ${\bf F}$, where ${\bf F}=\frac{\partial {\bf x}}{\partial {\bf X}}=\frac{\partial {\bf x}}{\partial {\bf x}_0} \frac{\partial {\bf x}_0}{\partial {\bf X}}$=$\nabla_0{\bf x}\nabla_{\bf X}{\bf x}_0$ is deformation tensor of the solid. In the above, $\nu^f$ and $\nu^s$ are viscosity of the fluid and solid respectively, $p^f$ and $p^s$ are pressure of the fluid and solid respectively.

The system is complemented with the following boundary and initial conditions.
\begin{equation}\label{interfaceBC1}
{\bf u}^f={\bf u}^s\quad on \quad  \Gamma_t,
\end{equation}
\begin{equation}\label{interfaceBC2}
{\bf n}^s{\bm \sigma}^f= {\bf n}^s{\bm \sigma}^s\quad on \quad  \Gamma_t,
\end{equation}
\begin{equation}\label{homogeneous_boundary}
{\bf u}^f={\bf 0}\quad on \quad  \Gamma,
\end{equation}
\begin{equation} \label{initialcd_fluid}
\left. {\bf u}^f\right|_{t=0}={\bf u}_0^f,
\end{equation}
\begin{equation} \label{initialcd_solid}
\left. {\bf u}^s\right|_{t=0}={\bf u}_0^s.
\end{equation}
Other boundary conditions are possible on $\Gamma$ but \cref{homogeneous_boundary} are used here for simplicity.
\section{Weak form on the continuous level}
\label{sec:wfotcl}
In the following context, let $L^2(\omega)$ be the square integrable functions in domain $\omega$, endowed with norm $\left\|u\right\|_{0,\omega}^2=\int_\omega \left|u\right|^2$ ($u\in L^2(\omega)$). Let $H^1(\omega)=\left\{u:u, \nabla u\in L^2(\omega)\right\}$ with the norm denoted by $\left\|u\right\|_{1,\omega}^2=\left\|u\right\|_{0,\omega}^2+\left\|\nabla u\right\|_{0,\omega}^2$. We also denote by $H_0^1(\omega)$ the subspace of $H^1(\omega)$ whose functions have zero values on the boundary of $\omega$, and denote by $L_0^2(\omega)$ the subspace of $L^2(\omega)$ whose functions have zero mean value.

Let
$
{p}=\left \{ 
\begin{matrix}
{{p}^f \quad in \quad \Omega^f} \\
{{p}^s \quad in \quad \Omega^s} \\
\end{matrix}\right.
$.
Given ${\bf v}\in H_0^1(\Omega)^d$, we perform the following symbolic operations:
$$
\int_{\Omega_t^f}{\rm Eq.}\cref{momentum_equation}\cdot{\bf v}d{\bf x}
+\int_{\Omega_t^s}{\rm Eq.}\cref{momentum_equation}\cdot{\bf v}d{\bf x}.
$$

Integrating the stress terms by parts, the above operations, using constitutive equation \cref{constitutive_fluid} and \cref{constitutive_solid} and boundary condition \cref{interfaceBC2}, gives:
\begin{equation}\label{operation1}
\begin{split}
&\rho^f\int_{\Omega}\frac{d{\bf u}}{dt} \cdot{\bf v}d{\bf x}
+\frac{\nu^f}{2}\int_{\Omega}{\rm D}{\bf u}:{\rm D}{\bf v}d{\bf x}
-\int_{\Omega}p\nabla \cdot {\bf v}d{\bf x} \\
&+\rho^{\delta}\int_{\Omega_t^s}\frac{d{\bf u}}{dt}\cdot{\bf v}d{\bf x}
+\frac{\nu^{\delta}}{2}\int_{\Omega_t^s}{\rm D}{\bf u}:{\rm D}{\bf v}d{\bf x}
+\int_{\Omega_t^s}{\bm\tau}^s:\nabla{\bf v}d{\bf x}
=0,
\end{split}
\end{equation}
where $\rho^{\delta}=\rho^s-\rho^f$, $\nu^{\delta}=\nu^s-\nu^f$. Note that the integrals on the interface $\Gamma_t$ are cancelled out using boundary condition \cref{interfaceBC2}. This is not surprising because they are internal forces for the whole FSI system considered here. 

Substitute the expression of ${{\bm\tau}^s}$ \cref{expression_of_taus} into \cref{operation1} and transfer the integral of the last term to the reference coordinate system. Then, the following symbolic operations for $q\in L^2(\Omega)$,
$$
-\int_{\Omega_t^f}{\rm Eq.}\cref{continuity_equation}qd{\bf x}
-\int_{\Omega_t^s}{\rm Eq.}\cref{continuity_equation}qd{\bf x},
$$ 
lead to the weak form of the FSI system as follows.

\begin{problem}\label{problem:continuous_weak_form}
Find ${\bf u}(t)\in H_0^1(\Omega)^d$ and $p(t) \in L_0^2(\Omega)$, such that
\begin{equation}\label{weak_form1}
\begin{split}
&\rho^f\int_{\Omega}\frac{\partial{\bf u}}{\partial t} \cdot{\bf v}d{\bf x}
+\rho^f \int_{\Omega}\left({\bf u}\cdot\nabla\right){\bf u}\cdot{\bf v}d{\bf x}
+\frac{\nu^f}{2}\int_{\Omega}{\rm D}{\bf u}:{\rm D}{\bf v}d{\bf x}
-\int_{\Omega}p\nabla \cdot {\bf v}d{\bf x} \\
&+\rho^{\delta}\int_{\Omega_t^s}\frac{\partial{\bf u}}{\partial t}\cdot{\bf v}d{\bf x} 
+\frac{\nu^{\delta}}{2}\int_{\Omega_t^s}{\rm D}{\bf u}:{\rm D}{\bf v}d{\bf x}
+\int_{\Omega_{\bf X}^s}{\bf P}({\bf F}):\nabla_{\bf X}{\bf v}d{\bf X}
=0,
\end{split}
\end{equation}
$\forall {\bf v}\in H_0^1(\Omega)^d$ and
\begin{equation}\label{weak_form2}
-\int_{\Omega} q\nabla \cdot {\bf u}d{\bf x}=0,
\end{equation}		
$\forall q \in L^2(\Omega)$.
\end{problem}

\begin{remark}
Because domain $\Omega$ is stationary (the Eulerian description will be used) and $\Omega_t^s$ is transient which will be updated by its own velocity (the updated Lagrangian description), there is a convection term from the total derivative of time in $\Omega$, but there is no convection term in $\Omega_t^s$.
\end{remark}

\section{Energy conservation on the continuous level}
A property of energy conservation for the weak forms \cref{weak_form1} and \cref{weak_form2} will be proved in this section.
\label{sec:ecotcl}
\begin{lemma}\label{lemma_potential_energy}
Assume the solid energy function $\Psi(\cdot)\in C^1$ over the set of second order tensors, then
\begin{equation}\label{potential_energy_relation}
\int_0^t\int_{\Omega_{\bf X}^s}{\bf P}:{\nabla}_{\bf X}{\bf u}d{\bf X}
=\int_{\Omega_{\bf X}^s}\Psi(\bf F)d{\bf X}
\end{equation}
\end{lemma}
\begin{proof}
Using the fact ${\bf A:C}=tr_{{\bf AC}^{\rm T}}$ (${\bf A}$ and ${\bf C}$ are arbitrary matrices), we have:
\begin{equation*}
\begin{split}
\frac{d}{dt}\int_{\Omega_{\bf X}^s}\Psi({\bf F})d{\bf X}
=\int_{\Omega_{\bf X}^s}\frac{\partial\Psi}{\partial{\bf F}}:\frac{d{\bf F}}{dt}d{\bf X}
=\int_{\Omega_{\bf X}^s}{\bf P}:\frac{d}{dt}\left({\bf I}+\nabla_{\bf X}{\bf d}\right)d{\bf X}
=\int_{\Omega_{\bf X}^s}{\bf P}:\nabla_{\bf X}{\bf u}d{\bf X},
\end{split}
\end{equation*}
where ${\bf d}$ is displacement of the solid at time $t$ in the above.
\end{proof}

\begin{lemma}\label{lemma_convection_zero}
If $\left({\bf u}, p\right)$ is the solution pair of \cref{problem:continuous_weak_form}, then
\begin{equation}\label{convection_zero}
\int_{\Omega}\left({\bf u}\cdot\nabla\right){\bf u}\cdot{\bf u}d{\bf x}=0.
\end{equation}
\end{lemma}
\begin{proof}
First,
\begin{equation*}
\int_{\Omega}\left({\bf u}\cdot\nabla\right){\bf u}\cdot{\bf u}d{\bf x}
=\int_{\Omega}\nabla\left({\bf u}\otimes{\bf u}\right)\cdot{\bf u}d{\bf x}
-\int_{\Omega}\left|{\bf u}\right|^2\nabla\cdot{\bf u}.
\end{equation*}
Integrate by parts:
\begin{equation*}
\int_{\Omega}\nabla\left({\bf u}\otimes{\bf u}\right)\cdot{\bf u}d{\bf x}
=\int_{\Gamma}\left|{\bf u}\right|^2{\bf u}\cdot {\bf n}
-\int_{\Omega}\left({\bf u}\cdot\nabla\right){\bf u}\cdot{\bf u}d{\bf x}.
\end{equation*}
According to the boundary condition \cref{homogeneous_boundary} and equation \cref{weak_form2}, we have
\begin{equation*}
\int_{\Gamma}\left|{\bf u}\right|^2{\bf u}\cdot {\bf n}	
=\int_{\Omega}\left|{\bf u}\right|^2\nabla\cdot{\bf u}
=0,
\end{equation*}
which further indicates equation \cref{convection_zero}.
\end{proof}

\begin{proposition} [Energy Conservation on the Continuous Level] Let $\left({\bf u}, p\right)$ be the solution pair of \cref{problem:continuous_weak_form}, then
\begin{equation}\label{energy balance}
\begin{split}
&\frac{\rho^f}{2}\int_{\Omega}|{\bf u}|^2d{\bf x}
+\frac{\nu^f}{2}\int_0^t\int_{\Omega}{\rm D}{\bf u}:{\rm D}{\bf u}d{\bf x} \\
+&\frac{\rho^\delta}{2}\int_{\Omega_t^s}|{\bf u}|^2d{\bf x}
+\frac{\nu^\delta}{2}\int_0^t\int_{\Omega_t^s}{\rm D}{\bf u}:{\rm D}{\bf u}d{\bf x}
+\int_{\Omega_{\bf X}^s}\Psi(\bf F)d{\bf X}=0.
\end{split}
\end{equation}
\end{proposition}
\begin{proof}
first choose ${\bf v}={\bf u}$ in \cref{weak_form1} and integrate from time $0$ to $t$, then let $q=p$ in \cref{weak_form2} and substitute into \cref{weak_form1}. Finally we can construct the above equation of energy balance due to \cref{lemma_potential_energy} and \cref{lemma_convection_zero}.
\end{proof}

\section{Weak form after discretization in time}
\label{sec:wfadit}
Using the Crank-Nicolson scheme to discretize equation \cref{weak_form1} in time, \cref{problem:continuous_weak_form} becomes:

\begin{problem}\label{problem_weak_after_time_discretization}
For each time step, find ${\bf u}_{n+1}\in H_0^1(\Omega)^d$ and $p_{n+1} \in L_0^2(\Omega)$, such that
\begin{equation}\label{weak_form1_discretization}
\begin{split}
&\rho^f\int_{\Omega}\frac{{\bf u}_{n+1}-{\bf u}_n}{\Delta t} \cdot{\bf v}d{\bf x}
+\rho^f \int_{\Omega}\left({\bf u}_*\cdot\nabla\right){\bf u}_*\cdot{\bf v}d{\bf x}
+\frac{\nu^f}{2}\int_{\Omega}{\rm D}{\bf u}_*:{\rm D}{\bf v}d{\bf x} \\
&-\int_{\Omega}p_{n+1}\nabla \cdot {\bf v}d{\bf x}
+\rho^{\delta}\int_{\Omega_{n+1}^s}\frac{{\bf u}_{n+1}-{\bf u}_n}{\Delta t} \cdot{\bf v}d{\bf x}
+\frac{\nu^{\delta}}{2}\int_{\Omega_{n+1}^s}{\rm D}{\bf u}_*:{\rm D}{\bf v}d{\bf x}\\
&+\int_{\Omega_{\bf X}^s}{\bf P}\left({\bf F}_*\right):\nabla_{\bf X}{\bf v}d{\bf X}
=O\left(\Delta t^2\right),
\end{split}
\end{equation}
$\forall{\bf v}\in H_0^1\left(\Omega\right)^d$ and 
\begin{equation}\label{weak_form2_discretization}
-\int_{\Omega} q\nabla \cdot {\bf u}_*d{\bf x}=0,
\end{equation}	
$\forall q\in L^2(\Omega)$, where ${\bf u}_*=\frac{{\bf u}_{n+1}+{\bf u}_n}{2}$ and ${\bf F}_*=\frac{{\bf F}_{n+1}+{\bf F}_n}{2}$.
\end{problem}

\begin{remark}\label{update_of_solid_mesh}
Notice that the subscript $n+1$ indicates that $\Omega_{n+1}^s$ is transient but $\Omega$ is stationary. $\Omega_{n+1}^s$ is updated from $\Omega_n^s$ by
$\Omega_{n+1}^s=\left\{{\bf x}_{n+1}:{\bf x}_{n+1}={\bf x}_n+{\bf u}_*\Delta t\right\}$, for all ${\bf x}_n\in \Omega_n^s$. 
\end{remark}

\section{Energy conservation after discretization in time}
\label{sec:ecadit}
Energy conservation for the weak form \cref{weak_form1_discretization} after time discretization will be proved in this section.
\begin{lemma}\label{energy_estimate1_lemma}
Assume $\Psi(\cdot)\in C^1$ over the set of second order tnesors, then
\begin{equation}\label{energy_estimate1}
\Delta t{\bf P}\left({\bf F}_*\right):\nabla_{\bf X}{\bf u}_*+O\left(\Delta t^3\right)
=\Psi({\bf F}_{n+1})-\Psi({\bf F}_n).
\end{equation}
\end{lemma}
\begin{proof}
Let $w(\xi)
=\Psi\left({\bf F}_n+\xi\nabla_{\bf X}{\bf u}_*\right)$, and notice that
\begin{equation*}
\Delta t\nabla_{\bf X}{\bf u}_*
=\nabla_{\bf X}{\bf x}_{n+1}-\nabla_{\bf X}{\bf x}_n
={\bf F}_{n+1}-{\bf F}_n,
\end{equation*}
then
\begin{equation*}
\begin{split}
&\Psi({\bf F}_{n+1})-\Psi({\bf F}_n) 
=w\left(\Delta t\right)-w(0)
=\Delta t w^\prime\left(\frac{\Delta t}{2}\right)+O\left(\Delta t^3\right).\\
\end{split}
\end{equation*}
Using the chain rule, \cref{energy_estimate1} holds thanks to
\begin{equation*}
w^\prime\left(\frac{\Delta t}{2}\right)
=\left.\frac{\partial\Psi}{\partial{\bf F}}\right|_{\xi=\frac{\Delta t}{2}}:\nabla_{\bf X}{\bf u}_*
={\bf P}\left({\bf F}_*\right):\nabla_{\bf X}{\bf u}_*.
\end{equation*}
\end{proof}

Similarly to \cref{lemma_convection_zero}, we have:
\begin{lemma}\label{lemma_convection_zero_discretization_in_time}
If $\left({\bf u}_{n+1}, p_{n+1}\right)$ is the solution pair of \cref{problem_weak_after_time_discretization}, then
\begin{equation}\label{convection_zero_discretization_in_time}
\int_{\Omega}\left({\bf u}_*\cdot\nabla\right){\bf u}_*\cdot{\bf u}_*d{\bf x}=0.
\end{equation}
\end{lemma}

\begin{proposition} [Local Energy Conservation]\label{lec} Let $\left({\bf u}_{n+1}, p_{n+1}\right)$ be the solution pair of \cref{problem_weak_after_time_discretization}, then
\begin{equation}\label{energy estimate_after_time_discretization}
\begin{split}
&\frac{\rho^f}{2}\int_{\Omega}\left|{\bf u}_{n+1}\right|^2d{\bf x}
+\frac{\rho^\delta}{2}\int_{\Omega_{n+1}^s}\left|{\bf u}_{n+1}\right|^2d{\bf x}
+\int_{\Omega_{\bf X}^s}\Psi\left({\bf F}_{n+1}\right)d{\bf X}  \\
&+\frac{\Delta t\nu^f}{2}\int_{\Omega}{\rm D}{\bf u}_*:{\rm D}{\bf u}_*d{\bf x}
+\frac{\Delta t\nu^\delta}{2}\int_{\Omega_{n+1}^s}{\rm D}{\bf u}_*:{\rm D}{\bf u}_*d{\bf x} \\
&= \frac{\rho^f}{2}\int_{\Omega}\left|{\bf u}_n\right|^2d{\bf x}
+\frac{\rho^\delta}{2}\int_{\Omega_n^s}\left|{\bf u}_n\right|^2d{\bf x}
+\int_{\Omega_{\bf X}^s}\Psi\left({\bf F}_n\right)d{\bf X}+O\left(\Delta t^3\right).
\end{split}
\end{equation}
\end{proposition}
\begin{proof}
First, let ${\bf v}={\bf u}_{n+1}$ in \cref{weak_form1_discretization} and $q=p_{n+1}$ in \cref{weak_form2_discretization}, to get the quation:
\begin{equation}\label{one_equation}
\begin{split}
&\rho^f\int_{\Omega}\left({\bf u}_{n+1}-{\bf u}_n\right)\cdot{\bf u}_{n+1}d{\bf x}
+\Delta t\rho^f \int_{\Omega}\left({\bf u}_*\cdot\nabla\right){\bf u}_*\cdot{\bf u}_{n+1}d{\bf x} \\
& +\frac{\Delta t\nu^f}{2}\int_{\Omega}{\rm D}{\bf u}_*:{\rm D}{\bf u}_{n+1}d{\bf x}
+\rho^{\delta}\int_{\Omega_{n+1}^s}\left({\bf u}_{n+1}-{\bf u}_n\right)\cdot{\bf u}_{n+1}d{\bf x} \\
&+\frac{\Delta t\nu^{\delta}}{2}\int_{\Omega_{n+1}^s}{\rm D}{\bf u}_*:{\rm D}{\bf u}_{n+1}d{\bf x}
+\Delta t\int_{\Omega_{\bf X}^s}{\bf P}\left({\bf F}_*\right):\nabla_{\bf X}{\bf u}_{n+1}d{\bf X}
=O\left(\Delta t^3\right),
\end{split}
\end{equation}
then let ${\bf v}={\bf u}_{n}$ in \cref{weak_form1_discretization} and $q=p_{n+1}$ in \cref{weak_form2_discretization}, to get the further equation:
\begin{equation}\label{another_equation}
\begin{split}
&\rho^f\int_{\Omega}\left({\bf u}_{n+1}-{\bf u}_n\right)\cdot{\bf u}_{n}d{\bf x}
+\Delta t\rho^f \int_{\Omega}\left({\bf u}_*\cdot\nabla\right){\bf u}_*\cdot{\bf u}_{n}d{\bf x} \\
& +\frac{\Delta t\nu^f}{2}\int_{\Omega}{\rm D}{\bf u}_*:{\rm D}{\bf u}_{n}d{\bf x}
+\rho^{\delta}\int_{\Omega_{n+1}^s}\left({\bf u}_{n+1}-{\bf u}_n\right)\cdot{\bf u}_{n}d{\bf x} \\
&+\frac{\Delta t\nu^{\delta}}{2}\int_{\Omega_{n+1}^s}{\rm D}{\bf u}_*:{\rm D}{\bf u}_{n}d{\bf x}
+\Delta t\int_{\Omega_{\bf X}^s}{\bf P}\left({\bf F}_*\right):\nabla_{\bf X}{\bf u}_{n}d{\bf X}
=O\left(\Delta t^3\right).
\end{split}
\end{equation}
		
Add the above two equations, we finally have \cref{lec} due to \cref{energy_estimate1_lemma} and \cref{lemma_convection_zero_discretization_in_time}.
\end{proof}

\begin{corollary} [Total Energy Conservation]\label{cor:tec}
Let $N+1=t/\Delta t$, where $t$ is the computational time, and use the following notation for the different contributions to the total energy at the time t:\\
\begin{itemize}
	\item Kinetic energy in $\Omega$:\qquad
\begin{equation}\label{Kinetic_energy}
	E_k(\Omega)=\frac{\rho^f}{2}\int_{\Omega}\left|{\bf u}_{N+1}\right|^2d{\bf x}.
\end{equation}
	\item Kinetic energy in $\Omega_t^s$:\qquad
\begin{equation}
	E_k(\Omega_t^s)=\frac{\rho^\delta}{2}\int_{\Omega_{N+1}^s}\left|{\bf u}_{N+1}\right|^2d{\bf x}.
\end{equation}
	\item Viscous dissipation in $\Omega$:\qquad
\begin{equation}\label{viscousf}
	E_d(\Omega)=\frac{\Delta t\nu^f}{2}\sum_{n=0}^{N}\int_{\Omega}{\rm D}{\bf u}_*:{\rm D}{\bf u}_*d{\bf x}.
\end{equation}
	\item Viscous dissipation in $\Omega_t^s$: \qquad
\begin{equation}\label{viscouss}
	E_d(\Omega_t^s)=\frac{\Delta t\nu^\delta}{2}\sum_{n=0}^{N}\int_{\Omega_{n+1}^s}{\rm D}{\bf u}_*:{\rm D}{\bf u}_*d{\bf x}.
\end{equation}
	\item Potential energy of solid:\qquad
\begin{equation}\label{Potential_energy_of_solid}
	E_p(\Omega_{\bf X}^s)=\int_{\Omega_{\bf X}^s}\Psi\left({\bf F}_{N+1}\right)d{\bf X}.
\end{equation}
\end{itemize}

Denote the total energy at time $t$ as:
\begin{equation}
E_{total}=E_k(\Omega)+E_k(\Omega_t^s)+E_d(\Omega)+E_d(\Omega_t^s)+E_p(\Omega_{\bf X}^s)
\end{equation}
and the error of total energy as:
\begin{equation}\label{energy estimate_after_time_discretization_closed}
Err=E_{total}- \frac{\rho^f}{2}\int_{\Omega}\left|{\bf u}_0\right|^2d{\bf x}
-\frac{\rho^\delta}{2}\int_{\Omega_0^s}\left|{\bf u}_0\right|^2d{\bf x}
-\int_{\Omega_{\bf X}^s}\Psi\left(\frac{\partial{\bf x}_0}{\partial {\bf X}}\right)d{\bf x},
\end{equation}
then,
\begin{equation*}
\left|Err\right|=O(\Delta t^2).
\end{equation*}
\end{corollary}
\begin{proof}
First let $O(\Delta t^3)=\sum_{k\ge 3}C_{n+1}^k\Delta t^k$ in \cref{energy estimate_after_time_discretization}, where $C_{n+1}^k$ is a constant dependent on the specific time step. Then add equation \cref{energy estimate_after_time_discretization} from $n=0$ to $n=N$, we have $Err=\sum_{k\ge 3}\Delta t^k\sum_{n=0}^{N}C_{n+1}^k$. \cref{energy estimate_after_time_discretization_closed} holds due to $\sum_{n=0}^{N}\left|C_{n+1}^k\right| \le C_{max}^kt/\Delta t$, where $C_{max}^k=max\{\left|C_{n+1}^k\right|, n=0,\cdots,N\}$.
\end{proof}

\begin{remark}\label{reduced_order1}
Note that, in subsequent sections, the viscous term in \cref{viscousf} and \cref{viscouss} will introduce an error $\Delta tO(h^m)$ ($m>0$) after discretization in space, where $h$ represents the mesh size. In that case it will be seen that the error of total energy will be reduced to $O(\Delta t)$ for a fixed mesh.	
\end{remark}	

\section{Analysis of the stationary problem corresponding to \cref{problem_weak_after_time_discretization}}
\label{sec:aotstp}
We now focus on the stationary problem corresponding to one step of the \cref{problem_weak_after_time_discretization} and consider its well-posedness and discretization in space. In order to make the analysis tractable we make the following simplifying assumptions, equivalent to those presented in \cite{Boffi_2016}.
\begin{itemize}
\item Neglect the convection term.
\item Assume $\rho^\delta\ge 0$ and $\nu^\delta\ge 0$.
\item Assume a linear model for ${\bf P}$, i.e., ${\bf P}=\mu^s{\bf F}$.
\end{itemize}

\begin{remark}
We have implemented numerical examples which include convection and the cases of $\rho^\delta< 0$ or $\nu^\delta< 0$ (see \cref{sec:numerical_exs}, and \cite{Wang_2017} for more numerical tests). The linear assumption for ${\bf P}$ is important in order to define the following bilinear form \cref{as()}, and we have only implemented an incompressible neo-Hookean model in this linear case. For non-linear cases, it may be that linearization and/or modification of \cref{problem_weak_after_time_discretization} are required. This will be a topic of our further investigation.
\end{remark}

Under the assumption of a linear model for ${\bf P}$, we have
\begin{equation}\label{compute_solid_at_reference}
\begin{split}
&\int_{\Omega_{\bf X}^s}{\bf P}\left({\bf F}_*\right):\nabla_{\bf X}{\bf v}d{\bf X}
=\mu^s\int_{\Omega_{\bf X}^s}\frac{{\bf F}_{n+1}+{\bf F}_n}{2}:\nabla_{\bf X}{\bf v}d{\bf X} \\
&=\mu^s\int_{\Omega_{\bf X}^s}{\bf F}_n:\nabla_{\bf X}{\bf v}d{\bf X}
+\frac{\mu^s\Delta t}{2}\int_{\Omega_{\bf X}^s}\nabla_{\bf X}{\bf u}_*:\nabla_{\bf X}{\bf v}d{\bf X}.
\end{split}
\end{equation}

Define the following bilinear forms:
\begin{equation}\label{af()}
a^f({\bf u},{\bf v})
=\alpha\int_\Omega{\bf u}\cdot{\bf v}d{\bf x}
+\frac{\nu^f}{2}\int_\Omega{\rm D}{\bf u}:{\rm D}{\bf v}d{\bf x},
\end{equation}
\begin{equation}\label{as()}
a^s({\bf u},{\bf v})
=\beta\int_{\Omega_{n+1}^s}{\bf u}\cdot{\bf v}d{\bf x}
+\frac{\nu^{\delta}}{2}\int_{\Omega_{n+1}^s}{\rm D}{\bf u}:{\rm D}{\bf v}d{\bf x}
+\gamma\int_{\Omega_{\bf X}^s}\nabla_{\bf X}{\bf u}:\nabla_{\bf X}{\bf v}d{\bf X},
\end{equation}
and
\begin{equation}\label{b()}
b({\bf v},q)=-\int_\Omega q\nabla\cdot{\bf v}d{\bf x},
\end{equation}
where $\alpha=2\rho^f/\Delta t$, $\beta=2\rho^{\delta}/\Delta t$ and $\gamma=\mu^s\Delta t/2$. We also define the following linear forms:
\begin{equation}\label{lf()}
\ell^f({\bf v})
=\frac{2\rho^f}{\Delta t}\int_\Omega{\bf u}_n\cdot{\bf v}d{\bf x},
\end{equation}
and
\begin{equation}\label{ls()}
\ell^s({\bf v})
=\frac{2\rho^\delta}{\Delta t}\int_{\Omega_{n+1}^s}{\bf u}_n\cdot{\bf v}d{\bf x}
-\mu^s\int_{\Omega_{\bf X}^s}{\bf F}_n:\nabla_{\bf X}{\bf v}d{\bf X}.
\end{equation}
Let
\begin{equation}\label{a()}
a\left({\bf u},{\bf v}\right)=a^f\left({\bf u},{\bf v}\right)+a^s\left({\bf u},{\bf v}\right),
\end{equation}
and
\begin{equation}\label{l()}
\ell({\bf v})=\ell^f({\bf v})+\ell^s({\bf v}),
\end{equation}
then the weak form corresponding to one step of \cref{problem_weak_after_time_discretization} with ${\bf u}={\bf u}_*$ and $p=p_{n+1}$, using the above notation and assumptions, can be stated as:

\begin{problem}\label{problem_bilinear_form}
Find ${\bf u}\in \mathbb{V}$, $p\in \mathbb{P}$, such that
\begin{equation*}
\left \{
\begin{tabular}{lr}
$a({\bf u},{\bf v})+b({\bf v},p)=\ell({\bf v})$ & $\forall{\bf v}\in \mathbb{V}$ \\
$b({\bf u},q)=0$ & $\forall q\in \mathbb{P}$
\end{tabular}
\right.
,
\end{equation*}
where $\mathbb{V}=H_0^1\left(\Omega\right)^d$ and $\mathbb{P}=L_0^2(\Omega)$.
\end{problem}

We shall use the following norms for vector functions and matrix functions respectively: $\left\|{\bf u}\right\|_{0,\omega}^2=\sum_{i=1}^{d}\left\|u_i\right\|_{0,\omega}^2$ and $\left\|{\bf A}\right\|_{0,\omega}^2=\sum_{i=1}^{d}\sum_{j=1}^{d}\left\|A_{ij}\right\|_{0,\omega}^2$, and denote $\|{\bf u}\|_{\mathbb{V}}=\|{\bf u}\|_{1,\Omega}$.

\subsection{Well-posedness of \cref{problem_bilinear_form}}

\begin{lemma}\label{coercivity_of_a}
$a({\bf u}, {\bf v})$ is coercive, i.e., there exists a positive real number $c_a$, such that for $\forall {\bf u}\in\mathbb{V}$,
\begin{equation}\label{coercivity_of_a_eq}
a\left({\bf u}, {\bf u}\right)
\ge c_a\left\|{\bf u}\right\|_\mathbb{V}^2.
\end{equation}
\end{lemma}
\begin{proof}
First, according to definitions of $\left\|{\rm D}{\bf u}\right\|_{0,\Omega}$ and $\left\|{\bf u}\right\|_{1,\Omega}$,
\begin{equation*}
\left\|{\rm D}{\bf u}\right\|_{0,\Omega}^2
=\left\|\nabla{\bf u}\right\|_{0,\Omega}^2
+\left\|\nabla^{\rm T}{\bf u}\right\|_{0,\Omega}^2
+2\left\|\nabla \cdot{\bf u}\right\|_{0,\Omega}^2
\ge
2\left\|\nabla{\bf u}\right\|_{0,\Omega}^2,
\end{equation*}
and
\begin{equation*}\label{proof_2_1}
\begin{split}
\left\|{\bf u}\right\|_{1,\Omega}^2
=\left\|{\bf u}\right\|_{0,\Omega}^2
+\left\|\nabla{\bf u}\right\|_{0,\Omega}^2
\le(C_\Omega^2+1)\left\|\nabla{\bf u}\right\|_{0,\Omega}^2
\le\frac{C_\Omega^2+1}{2}\left\|{\rm D}{\bf u}\right\|_{0,\Omega}^2,
\end{split}
\end{equation*}
the above estimate is due to the Poincar$\rm\acute{e}$ inequality $\left\|{\bf u}\right\|_{0,\Omega}\le C_\Omega\left\|\nabla{\bf u}\right\|_{0,\Omega} $, where $C_\Omega$ is constant. Then, 
\begin{equation*}\label{proof_2_2}
\begin{split}
& a({\bf u}, {\bf u})
=\alpha\|{\bf u}\|_{0,\Omega}^2
+\frac{\nu^f}{2}\|{\rm D}{\bf u}\|_{0,\Omega}^2 
+\beta\|{\bf u}\|_{0,\Omega_{n+1}^s}^2 \\
&+\frac{\nu^{\delta}}{2}\|{\rm D}{\bf u}\|_{0,\Omega_{n+1}^s}^2
+\gamma\|{\nabla_{\bf X}}{\bf u}\|_{0,\Omega_{\bf X}^s}^2 
\ge \frac{\nu^f}{C_\Omega^2+1}\|{\bf u}\|_{1,\Omega}^2,
\end{split}
\end{equation*}
using our assumption $\beta=2\rho^{\delta}/\Delta t\ge 0$ and $\nu^{\delta}\ge 0$.
\end{proof}

\begin{lemma}\label{continuity_of_a}
$a({\bf u}, {\bf v})$ is bounded, i.e., there exists a positive real number $C_a$, such that for $\forall {\bf u}, {\bf v}\in {\mathbb{V}}$,
\begin{equation}\label{continuity_of_a_eq}
\left|a({\bf u}, {\bf v})\right|
\le
C_a\left\|{\bf u}\right\|_\mathbb{V}\left\|{\bf v}\right\|_\mathbb{V}.
\end{equation}
\end{lemma}
\begin{proof}
First
\begin{equation}
\begin{split}\label{proof_1_1}
& a({\bf u},{\bf v})
\le \alpha\left\|{\bf u}\right\|_{0,\Omega}\left\|{\bf v}\right\|_{0,\Omega}
+ \frac{\nu^f}{2}\left\|{\rm D}{\bf u}\right\|_{0,\Omega}\left\|{\rm D}{\bf v}\right\|_{0,\Omega} \\
&+\beta\left\|{\bf u}\right\|_{0,\Omega_{n+1}^s}\left\|{\bf v}\right\|_{0,\Omega_{n+1}^s}
+ \frac{\nu^{\delta}}{2}\left\|{\rm D}{\bf u}\right\|_{0,\Omega_{n+1}^s}\left\|{\rm D}{\bf v}\right\|_{0,\Omega_{n+1}^s} \\
&+ \gamma\left\|{\nabla_{\bf X}}{\bf u}\right\|_{0,\Omega_{\bf X}^s}\left\|{\nabla_{\bf X}}{\bf v}\right\|_{0,\Omega_{\bf X}^s}.
\end{split}
\end{equation}
Noticing that
\begin{equation}\label{proof_1_2}
\left\|\nabla_{\bf X}{\bf u}\right\|_{0,\Omega_{\bf X}^s}
\le
\left\|{\bf u}\right\|_{1,\Omega_{\bf X}^s}
\le
\left\|{\bf u}\right\|_{1,\Omega},
\end{equation}
and
\begin{equation*}
\left\|{\rm D}{\bf u}\right\|_{0,\Omega}
\le2\left\|{\nabla}{\bf u}\right\|_{0,\Omega}
\le 2\left\|{\bf u}\right\|_{1,\Omega},
\end{equation*}
we can observe \cref{continuity_of_a_eq} holds from \cref{proof_1_1}, with $C_a=\alpha+\beta+2\nu^f+2\nu^\delta+\gamma$.
\end{proof}

\begin{lemma}\label{continuity_of_f}
$\ell({\bf v})$ is bounded, i.e., there exists a positive real number $C_\ell$, such that for $\forall {\bf v}\in{\mathbb{V}}$,
\begin{equation}\label{continuity_of_f_eq}
\left|\ell({\bf v})\right|\le C_\ell\left\|{\bf v}\right\|_\mathbb{V}.
\end{equation}
\end{lemma}
\begin{proof}
\begin{equation*}\label{proof_3_1}
\begin{split}
& \ell({\bf v})
\le \frac{2\rho^f}{\Delta t}\left\|{\bf u}_n\right\|_{0,\Omega}\left\|{\bf v}\right\|_{0,\Omega}
+\frac{2\rho^{\delta}}{\Delta t}\left\|{\bf u}_n\right\|_{0,\Omega_{n+1}^s}\left\|{\bf v}\right\|_{0,\Omega_{n+1}^s}
+\mu^s\left\|{\bf F}_n\right\|_{0,\Omega_{\bf X}^s} \left\|\nabla_{\bf X}{\bf v}\right\|_{0,\Omega_{\bf X}^s} \\
& \le 
\frac{2\rho^s}{\Delta t}\left\|{\bf u}_n\right\|_{0,\Omega}\left\|{\bf v}\right\|_{0,\Omega}
+\mu^s\left\|\nabla_{\bf X}{\bf x}_n\right\|_{0,\Omega_{\bf X}^s} \left\|\nabla_{\bf X}{\bf v}\right\|_{0,\Omega_{\bf X}^s}.
\end{split}
\end{equation*}
Notice that $\left\|\nabla_{\bf X}{\bf x}_n\right\|_{0,\Omega_{\bf X}^s}$ is bounded owing to the assumption that ${\bf x}={\bf x}({\bf s},t)$ (${\bf s}\in \Omega_{\bf X}^s$) is Lipschitz invertible. We further get \cref{continuity_of_f_eq} using relation \cref{proof_1_2}. 
\end{proof} 

\begin{proposition}[Inf-Sup Condition] There exist $c_b>0$ such that
\begin{equation}
\inf_{q\in\mathbb{P}}\sup_{{\bf v}\in\mathbb{V}}\frac{b({\bf v},q)}{\|{\bf v}\|_{\mathbb{V}}\|q\|_{\mathbb{P}}}
\ge c_b.
\end{equation}
\end{proposition}
\begin{proof}
can be found in \cite [Section 12.2] {brenner2007mathematical}.
\end{proof}

The above three lemmas and the inf-sup condition imply \cite [Lemma 12.2.12] {brenner2007mathematical}:
\begin{theorem}
\cref{problem_bilinear_form} has a unique solution $({\bf u}, p)\in \mathbb{V}\times\mathbb{P}$.
\end{theorem}

\subsection{Discretization in space}\label{Discretization in space}
We shall use a fixed Eulerian mesh for $\Omega$ and an updated Lagrangian mesh for $\Omega_{n+1}^s$ to discretize \cref{problem_bilinear_form}. First, we discretize $\Omega$ as $\Omega^h$ with the corresponding finite element spaces as
$$
V^h(\Omega^h)=span\left\{\varphi_1,\cdots,\varphi_{N^u}\right\} \subset H_0^1\left(\Omega\right)
$$
and
$$
L^h(\Omega^h)=span\left\{\phi_1,\cdots,\phi_{N^p}\right\} \subset L_0^2\left(\Omega\right).
$$
The approximated solution ${\bf u}^h$ and $p^h$ can be expressed in terms of these basis functions as
\begin{equation}\label{uh}
{\bf u}^h({\bf x})=\sum_{i=1}^{N^u}{\bf u}({\bf x}_i)\varphi_i({\bf x}), \quad
p^h({\bf x})=\sum_{i=1}^{N^p}p({\bf x}_i)\phi_i({\bf x}).
\end{equation}

We further discretize $\Omega_{n+1}^s$ as $\Omega_{n+1}^{sh}$ with the corresponding finite element spaces as:
$$
V^{sh}(\Omega_{n+1}^{sh})=span\left\{\varphi_1^s,\cdots,\varphi_{N^s}^s\right\} \subset H^1\left(\Omega_{n+1}^s\right),
$$
and approximate $\left.{\bf u}^h({\bf x})\right|_{{\bf x}\in\Omega_{n+1}^{sh}}$ as:
\begin{equation}\label{ush}
{\bf u}^{sh}\left({\bf x}\right)
=\sum_{i=1}^{N^s}{\bf u}^h({\bf x}_i^s)\varphi_i^s({\bf x})
=\sum_{i=1}^{N^s}\sum_{j=1}^{N^u}{\bf u}({\bf x}_j)\varphi_j({\bf x}_i^s)\varphi_i^s({\bf x}),
\end{equation}
where ${\bf x}_i^s$ is the nodal coordinate of the solid mesh. Notice that the above approximation defines an $L^2$ projection $P_{n+1}$ from $V^h\left(\Omega^h\right)^d$ to $V^{sh}\left(\Omega_{n+1}^{sh}\right)^d$: $P_{n+1}\left({\bf u}^h({\bf x})\right)={\bf u}^{sh}\left({\bf x}\right)$.

\begin{remark}
The solid domain $\Omega_t^s$ is actually discretized once on $\Omega_0^s$, and then $\Omega_n^s (n\ge 1)$ is updated from the mesh at the previous time step as illustrated in \cref{update_of_solid_mesh}.
\end{remark}

We then discretize \cref{problem_bilinear_form} as follows.
\begin{problem}\label{discretized_problem_bilinear_form}
Find ${\bf u}^h\in \mathbb{V}^h$, $p^h\in \mathbb{P}^h$, such that
\begin{equation*}
\left \{
\begin{tabular}{lr}
$a^h({\bf u}^h,{\bf v}^h)+b({\bf v}^h,p^h)=\ell^h({\bf v}^h)$ & $\forall{\bf v}^h\in \mathbb{V}^h$ \\
$b({\bf u}^h,q^h)=0$ & $\forall q^h\in \mathbb{P}^h$
\end{tabular}
\right.
,
\end{equation*}
where $\mathbb{V}^h=V^h(\Omega^h)^d$, $\mathbb{P}^h=L^h(\Omega^h)$, $a^h$ can be expressed as:
\begin{equation}\label{ah()}
a^h\left({\bf u}^h,{\bf v}^h\right)=a^f\left({\bf u}^h,{\bf v}^h\right)+a^s\left({\bf u}^{sh},{\bf v}^{sh}\right),
\end{equation}
and $\ell^h$ can be expressed as:
\begin{equation}
\ell^h({\bf v}^h)\label{lh()}
=\ell^f({\bf v}^h)+\ell^s({\bf v}^{sh}).
\end{equation}	
where ${\bf u}^{sh}=P_{n+1}\left({\bf u}^h\right)$ and ${\bf v}^{sh}=P_{n+1}\left({\bf v}^h\right)$.
\end{problem}

From the error estimate for interpolation \cite [Chapter 4, Corollary 4.4.24] {brenner2007mathematical}, the approximation \cref{ush} can be bounded by:
\begin{equation}\label{interpolation_erro_estimate}
\left\| P_{n+1}\left({\bf u}^h({\bf x})\right)-{\bf u}^h({\bf x}) \right\|_{1,\Omega_{n+1}^s}
\le
C_h\left\|\nabla{\bf u}^h({\bf x})\right\|_{0,\Omega_{n+1}^s},
\end{equation}
where $C_h$ is a constant depending on the element of $\Omega^{sh}$.
In order to prove \cref{discretized_problem_bilinear_form} is well-posed, we still need to prove the continuity of $a^h({\bf u}^h,{\bf v}^h)$ and $\ell^h({\bf v}^h)$, and the coercivity of $a^h({\bf u}^h,{\bf v}^h)$. The main idea is to bound ${\bf u}^{sh}$, $\nabla_{\bf X}{\bf u}^{sh}$ and ${\rm D}{\bf u}^{sh}$ in $\mathbb{V}$, which can be achieved by the above interpolation error estimate.

\begin{lemma}
There exists a positive real number $c_1$ such that
\begin{equation}\label{estimate_of_ush}
\left\|{\bf u}^{sh}({\bf x})\right\|_{0,\Omega_{n+1}^s}
\le
c_1\left\|{\bf u}^h({\bf x}) \right\|_{1,\Omega}.
\end{equation}
\end{lemma}
\begin{proof}
From \cref{interpolation_erro_estimate}, we have:
\begin{equation}\label{estimate_from_interpolation}
\begin{split}
&\left\|{\bf u}^{sh}({\bf x})\right\|_{1,\Omega_{n+1}^s}
=\left\|P_{n+1}\left({\bf u}^h({\bf x})\right)\right\|_{1,\Omega_{n+1}^s}\\
&\le \left\|{\bf u}^h({\bf x}) \right\|_{1,\Omega_{n+1}^s}
+\left\| P_{n+1}\left({\bf u}^h({\bf x})\right)-{\bf u}^h({\bf x}) \right\|_{1,\Omega_{n+1}^s} \\
& \le \left\|{\bf u}^h({\bf x}) \right\|_{1,\Omega_{n+1}^s}
+C_h\left\|\nabla{\bf u}^h({\bf x})\right\|_{0,\Omega_{n+1}^s}\\
&\le (1+C_h)\left\|{\bf u}^h({\bf x}) \right\|_{1,\Omega}.
\end{split}
\end{equation}
We have \cref{estimate_of_ush} due to $\left\|{\bf u}^{sh}({\bf x})\right\|_{0,\Omega_{n+1}^s}
\le \left\|{\bf u}^{sh}({\bf x})\right\|_{1,\Omega_{n+1}^s}$.
\end{proof}

\begin{lemma}
There exists a positive real number $c_2$ such that
\begin{equation}
\left\|\nabla_{\bf X}{\bf u}^{sh}\right\|_{0,\Omega_{\bf X}^s}
\le
c_2\left\|{\bf u}^h\right\|_{1,\Omega}.
\end{equation}
\end{lemma}
\begin{proof}
Using \cref{estimate_from_interpolation} we have:
\begin{equation*}\label{proof_1_2_discretization}
\begin{split}
&\left\|\nabla_{\bf X}{\bf u}^{sh}\right\|_{0,\Omega_{\bf X}^s}
=\left\|J^{-1}\nabla_{\bf X}{\bf u}^{sh}\right\|_{0,\Omega_{n+1}^s} 
=\left\|J^{-1}{\bf F}\nabla {\bf u}^{sh}\right\|_{0,\Omega_{n+1}^s}\\
&\le
\left\|J^{-1}{\bf F}\right\|_{0,\Omega_{n+1}^s}
\left\|\nabla{\bf u}^{sh}\right\|_{0,\Omega_{n+1}^s} 
=\left\|{\bf F}\right\|_{0,\Omega_{\bf X}^s}
\left\|\nabla{\bf u}^{sh}\right\|_{0,\Omega_{n+1}^s} \\
&\le
C_F\left\|\nabla{\bf u}^{sh}\right\|_{0,\Omega_{n+1}^s}
\le
C_F\left\|{\bf u}^{sh}\right\|_{1,\Omega_{n+1}^s}
\le
C_F\left(1+C_h\right)\left\|{\bf u}^h\right\|_{1,\Omega}.
\end{split}
\end{equation*}
In the above, $\left\|{\bf F}\right\|_{0,\Omega_{\bf X}^s}=\left\|\nabla_{\bf X}{\bf x}_{n+1}\right\|_{0,\Omega_{\bf X}^s}
\le C_F$ thanks to the Lipschitz inversible assumption of ${\bf x}={\bf x}\left({\bf s}, t\right)$ (${\bf s}\in\Omega_{\bf X}^s$).
\end{proof}

\begin{lemma}
There exists a positive real number $c_3$ such that
\begin{equation}
\left\|{\rm D}{\bf u}^{sh}\right\|_{0,\Omega_{n+1}^s}
\le
c_3\left\|{\bf u}^h\right\|_{1,\Omega}.
\end{equation}
\end{lemma}
\begin{proof}
Using \cref{estimate_from_interpolation} we have:
\begin{equation*}
\left\|{\rm D}{\bf u}^{sh}\right\|_{0,\Omega_{n+1}^s}
\le
2\left\|\nabla{\bf u}^{sh}\right\|_{0,\Omega_{n+1}^s}
\le
2\left\|{\bf u}^{sh}\right\|_{1,\Omega_{n+1}^s}
\le
2\left(1+C_h\right)\left\|{\bf u}^h\right\|_{1,\Omega}.
\end{equation*}
\end{proof}

From the definition of $a^s({\bf u}, {\bf v})$ \cref{as()} and $\ell^s({\bf v})$ \cref{ls()}, using the above three lemmas, we have the following two corollaries.
\begin{corollary}\label{continuity_of_as}
$a^s({\bf u}^{sh},{\bf v}^{sh})$ is bounded in $\mathbb{V}^h$, i.e., there exists a positive real number $C_{s}^h$, such that for $\forall {\bf u}^h, {\bf v}^h\in {\mathbb{V}^h}$,
\begin{equation}\label{continuity_of_as_eq}
\left|a^s({\bf u}^{sh}, {\bf v}^{sh})\right|
\le
C_s^h\left\|{\bf u}^h\right\|_\mathbb{V}\left\|{\bf v}^h\right\|_\mathbb{V}.
\end{equation}
\end{corollary}

\begin{corollary}\label{continuity_of_ls}
	$\ell^s({\bf v}^{sh})$ is bounded in $\mathbb{V}^h$, i.e., there exists a positive real number $C_{\ell}^h$, such that for $\forall {\bf v}^h\in {\mathbb{V}^h}$,
	\begin{equation}\label{continuity_of_ls_eq}
	\left|\ell^s({\bf v}^{sh})\right|
	\le
	C_\ell^h\left\|{\bf v}^h\right\|_\mathbb{V}.
	\end{equation}
\end{corollary}

\begin{lemma}\label{coercivity_of_ah}
$a^h({\bf u}^h, {\bf v}^h)$ is coercive, i.e., there exists a positive real number $c_a^h$, such that for $\forall {\bf u}^h\in\mathbb{V}^h$,
\begin{equation}\label{coercivity_of_ah_eq}
a^h\left({\bf u}^h, {\bf u}^h\right)
\ge c_a^h\left\|{\bf u}^h\right\|_\mathbb{V}^2.
\end{equation}
\end{lemma}
\begin{proof}
This proof follows the same procedure as the proof for \cref{coercivity_of_a} by changing $a({\bf u}, {\bf u})$ to $a^h({\bf u}^h, {\bf u}^h)$.
\end{proof}

\begin{lemma}\label{continuity_of_ah}
	$a^h({\bf u}^h, {\bf v}^h)$ is bounded, i.e., there exists a positive real number $C_a^h$, such that for $\forall {\bf u}^h, {\bf v}^h\in {\mathbb{V}^h}$,
	\begin{equation}\label{continuity_of_ah_eq}
	\left|a^h({\bf u}^h, {\bf v}^h)\right|
	\le
	C_a^h\left\|{\bf u}^h\right\|_\mathbb{V}\left\|{\bf v}^h\right\|_\mathbb{V}.
	\end{equation}
\end{lemma}
\begin{proof}
By the Cauchy-Schwarz inequality, $a^f({\bf u}^h, {\bf v}^h)$ is a bounded bilinear functional. Using the definition \cref{ah()}, $a^h({\bf u}^h, {\bf v}^h)$ is bounded due to \cref{continuity_of_as_eq}.
\end{proof}

\begin{lemma}\label{continuity_of_lh}
	$\ell^h({\bf v}^h)$ is bounded, i.e., there exists a positive real number $C_\ell^h$, such that for $\forall {\bf v}^h\in{\mathbb{V}^h}$,
	\begin{equation}\label{continuity_of_lh_eq}
	\left|\ell^h({\bf v}^h)\right|\le C_\ell^h\left\|{\bf v}^h\right\|_\mathbb{V}.
	\end{equation}
\end{lemma}
\begin{proof}
By the Cauchy-Schwarz inequality, $\ell^f({\bf v}^h)$ is a bounded linear functional. According to the definition \cref{lh()}, $\ell^h({\bf v}^h)$ is bounded due to \cref{continuity_of_ls_eq}.
\end{proof} 

We choose the $P_2/P_1$ or $P_2/(P_1+P_0)$ element which satisfies the following discrete inf-sup condition \cite{Boffi_2011}\cite [Section 12.6] {brenner2007mathematical}:

\begin{proposition}[Discrete Inf-Sup Condition] There exist $c_b^h>0$ such that
	\begin{equation}
	\inf_{q^h\in\mathbb{P}^h}\sup_{{\bf v}^h\in\mathbb{V}^h}\frac{b({\bf v}^h,q^h)}{\|{\bf v}^h\|_{\mathbb{V}}\|q^h\|_{\mathbb{P}}}
	\ge c_b^h.
	\end{equation}
\end{proposition}

Using the above three lemmas (\cref{coercivity_of_ah} - \ref{continuity_of_lh}) and the discrete inf-sup condition, we further have the following optimal error estimate result \cite [Corollary 12.5.18] {brenner2007mathematical}.

\begin{theorem}
	Let $({\bf u},p)$ and $({\bf u}^h,p^h)$ be the solution pairs of \cref{problem_bilinear_form} and \cref{discretized_problem_bilinear_form} respectively, then there is a constant $c$ depending on $C_a^h$, $c_a^h$ and $c_b^h$ such that
	\begin{equation}
	\left\|{\bf u}-{\bf u}^h\right\|_{\mathbb{V}}
	+\left\|p-p^h\right\|_{\mathbb{P}}
	\le
	c\left(\inf_{{\bf v}\in\mathbb{V}^h}\left\|{\bf u}-{\bf v}\right\|_{\mathbb{V}}+
	\inf_{q\in\mathbb{P}^h}\left\|p-q\right\|_{\mathbb{P}}\right).
	\end{equation}
\end{theorem}

\section{Implementation}
\label{sec:implementation}

We shall use an incompressible neo-Hookean model as described in section \ref{subsec:neo-hookean}. The treatment of convection is presented in section \ref{subsec:convection}, and the preconditioned iterative solver is introduced in section \ref{subse:solver}.
\subsection{Neo-Hookean hyperelastic solid}
\label{subsec:neo-hookean}
For an incompressible neo-Hookean hyperelastic material, the energy function is described as:
\begin{equation}\label{expression_of_psi}
\Psi\left({\bf F}\right)=\frac{\mu^s}{2}\left(tr_{{\bf F}^{\rm T}{\bf F}}-d\right).
\end{equation}
Notice that $tr_{{\bf F}^{\rm T}{\bf F}}=tr_{{F}_{ki}{F}_{kj}}={F}_{ki}^2$,
therefore $\frac{\partial tr_{{\bf F}^{\rm T}{\bf F}}}{\partial{\bf F}}
=\frac{\partial}{\partial {F}_{mn}}\left({F}_{ki}^2\right)=2{\bf F}$ and ${\bf P}=\mu^s{\bf F}$, hence according to (\ref{expression_of_taus}),
\begin{equation}\label{expression_of_taus_B}
{\bm\tau}^s=J^{-1}\mu^s{\bf B},
\end{equation}
where ${\bf B}={\bf F}{\bf F}^{\rm T}$ is the left Cauchy-Green deformation tensor.

\begin{remark} It is convenient to choose the reference configuration $({\bf B}={\bf I})$ the same as the initial configuration if the solid is initially stress free. In this case, $J\equiv 1$, and if both the fluid and solid are initially stationary, it can be seen from (\ref{constitutive_solid}) and (\ref{expression_of_taus_B}) that the solid is subjected to a hydrostatic stress $\mu^s-p^s$, while the fluid is subject to a hydrostatic stress $-p^f$ due to (\ref{constitutive_fluid}). From the boundary condition (\ref{interfaceBC2}), we can observe that $\mu^s-p^s=-p^f$ or  $p^s-p^f=\mu^s$. In order to avoid this unnecessary jump of pressure across the interface, we shall replace (\ref{expression_of_taus_B}) as:
\begin{equation}\label{expression_of_tao_add_I}
{\bm\tau}^s=J^{-1}\mu^s\left({\bf B}-{\bf I}\right).
\end{equation}
This actually does not influence the solid constitutive equation (\ref{constitutive_solid}), because $J^{-1}\mu^s{\bf I}$ can be integrated into the pressure term. However, this is important for a numerical scheme based on unfitted meshes as used in this article, where difficult to accurately capture the jump of pressure across the fluid-solid interface. This also does not change analysis of the energy-conservation and well-posedness for the proposed scheme.
\end{remark}

Expressing ${\bf F}$ by velocity as shown in \cref{compute_solid_at_reference}, then the weak form \cref{weak_form1_discretization} becomes:
\begin{equation}\label{weak_form1_discretization_velocity}
\begin{split}
&\rho^f\int_{\Omega}\frac{{\bf u}_{n+1}-{\bf u}_n}{\Delta t} \cdot{\bf v}d{\bf x}
+\rho^f \int_{\Omega}\left({\bf u}_*\cdot\nabla\right){\bf u}_*\cdot{\bf v}d{\bf x}
+\frac{\nu^f}{2}\int_{\Omega}{\rm D}{\bf u}_*:{\rm D}{\bf v}d{\bf x} \\
&-\int_{\Omega}p_{n+1}\nabla \cdot {\bf v}d{\bf x}
+\rho^{\delta}\int_{\Omega_{n+1}^s}\frac{{\bf u}_{n+1}-{\bf u}_n}{\Delta t} \cdot{\bf v}d{\bf x}
+\frac{\nu^{\delta}}{2}\int_{\Omega_{n+1}^s}{\rm D}{\bf u}_*:{\rm D}{\bf v}d{\bf x} \\
&-\mu^s\int_{\Omega_{n+1}^s}J_{n+1}^{-1}\nabla\cdot{\bf v}d{\bf x}
+\frac{\mu^s\Delta t}{2}\int_{\Omega_{\bf X}^s}\nabla_{\bf X}{\bf u}_*:\nabla_{\bf X}{\bf v}d{\bf X} \\
&=-\mu^s\int_{\Omega_{\bf X}^s}{\bf F}_n:\nabla_{\bf X}{\bf v}d{\bf X}+O\left(\Delta t^2\right).
\end{split}
\end{equation}

\begin{remark}\label{reduced_order2}
Like noted in \cref{reduced_order1}, term $\frac{\mu^s\Delta t}{2}\int_{\Omega_{\bf X}^s}\nabla_{\bf X}{\bf u}_*:\nabla_{\bf X}{\bf v}d{\bf X}$ in above equation will introduce an error $\Delta tO(h^m)$ ($m>0$) after discretization in space, which then reduce the order to $O(\Delta t)$ for a fixed mesh.	
\end{remark}

\subsection{Treatment of convection}
\label{subsec:convection}
Without considering the non-linear convection term, equation \cref{weak_form1_discretization_velocity} is almost a linear equation except the moving domain $\Omega_{n+1}^s$. So we have to iteratively construct $\Omega_{n+1}^s$ and take derivative on it. For the low Reynolds number ($Re<50$) that we consider in this article, the convection term can also be arranged on the right-hand side of the equation and included in one iteration loop, i.e., a fixed point iteration is adopted to solve the nonlinear system at each time step. For other methods to treat convection, readers may refer to \cite{pironneau1989finite, Zienkiewic2014}. Notice that no artificial diffusion terms are added in this case, and we can measure the system energy exactly using the energy functions defined from \cref{Kinetic_energy} to \cref{Potential_energy_of_solid}.

\begin{remark}
$J_{n+1}^{-1}=1/det_{{\bf F}_{n+1}}$ makes the term $\mu^s\int_{\Omega_{n+1}^s}J_{n+1}^{-1}\nabla\cdot{\bf v}d{\bf x}$ non-linear in equation \cref{weak_form1_discretization_velocity}. This term is also iteratively solved together with the convection term in one loop.
\end{remark}

\subsection{Iterative linear algebra solver}
\label{subse:solver}

\cref{discretized_problem_bilinear_form} leads to the following linear equation system \cite{Wang_2017}:
\begin{equation}\label{lastlinerequation}
\begin{bmatrix}
{\bf A} & {\bf B} \\
{\bf B}^{\rm T} &{\bf 0}
\end{bmatrix}
\begin{pmatrix}
{\bf u} \\
{\bf p}
\end{pmatrix}
=
\begin{pmatrix}
{\bf b} \\
{\bf 0}
\end{pmatrix}
,
\end{equation}
where
\begin{equation}\label{matirx_A}
{\bf A}={\bf M}/\Delta t+{\bf K}+{\bf D}^{\rm T}\left({\bf M}^s/\Delta t+{\bf K}^s\right){\bf D},
\end{equation}
and
\begin{equation}\label{forcevector_b}
{\bf b}={\bf f}+{\bf D}^{\rm T}{\bf f}^s+{\bf M}{\bf u}^*/\Delta t+{\bf D}^{\rm T}{\bf M}^s{\bf D}{\bf u}^n/\Delta t.
\end{equation}

In the above, matrix ${\bf D}$ is the isoparametric interpolation matrix derived from equation (\ref{ush}) which can be expressed as
\begin{equation*}
	{\bf D}=
	\begin{bmatrix}
		{\bf P}^{\rm T} & {\bf 0} \\
		{\bf 0} & {\bf P}^{\rm T} \\
	\end{bmatrix}
	,
	{\bf P}_{ij}=\varphi_i({\bf x}_j^s)
	.
\end{equation*}
All the other matrices and vectors arise from standard FEM discretization: ${\bf M}$ and ${\bf M}^s$ are mass matrices from discretization of the first term of $a^f({\bf u}^h,{\bf v}^h)$ and the first term of $a^s({\bf u}^{sh},{\bf v}^{sh})$ respectively, and similarly the stiffness matrices ${\bf K}$ and ${\bf K}^s$ are from the last term of $a^f({\bf u}^h,{\bf v}^h)$ and the last two terms of $a^s({\bf u}^{sh},{\bf v}^{sh})$ respectively. ${\bf B}$ is from discretization of the linear functional $b\left({\bf v}^h,p^h\right)$. The force vectors ${\bf f}$ and ${\bf f}^s$ come from discretization of $\ell^f({\bf u}^h)$ and $\ell^s({\bf u}^{sh})$ respectively.

We use the block matrix
$
\begin{bmatrix}
	{{\bf M/}\Delta t + {\bf K}} & {\bf 0} \\
	{\bf 0}^{\rm T} &{{\bf M/}\Delta t}
\end{bmatrix}
$
as a preconditioner, computed by an incomplete Cholesky decomposition, and MinRes algorithm \cite{Elman_2014} to solve equation \cref{lastlinerequation}. A stable convergence performance can be observed although this is not the topic of this article.

\section{Numerical experiments}
\label{sec:numerical_exs}
In this section, we focus on the validation of energy conservation of the proposed numerical method in two and three dimensions. For more two-dimensional numerical examples and validation see \cite{Wang_2017}. The $P_2/(P_1+P_0)$ elements will be used, i.e., the standard $Taylor$-$Hood$ elements $P_2P_1$ is enriched by a constant $P_0$ for approximation of the pressure. This element has the property of local mass conservation and the constant $P_0$ may better capture the element-based jump of pressure \cite{Arnold_2002, Boffi_2011}. We shall demonstrate the improvement of mass conservation and energy conservation by using the $P_2/(P_1+P_0)$ elements compared to the $P_2P_1$ elements. We shall also compare the Crank-Nicolson scheme and the backward Euler scheme (see \cref{appendix:backward-Euler}) in this section.

\subsection{Oscillating disc driven by an initial kinetic energy}
\label{subsec:oddbaike}
In this test, we consider an enclosed flow (${\bf n}\cdot{\bf u}=0$) in $\Omega=[0,1]\times[0,1]$ with a periodic boundary condition. A solid disc is initially located in the middle of the square $\Omega$ and has a radius of $0.2$. The initial velocity of the fluid and solid are prescribed by the following stream function
\begin{equation*}
\Psi=\Psi_0{\rm sin}(ax){\rm sin}(by),
\end{equation*}
where $\Psi_0=5.0\times10^{-2}$ and $a=b=2\pi$. Two parameter sets (\cref{Parameter sets for test problem of the oscillating disc}) are used to test the energy and mass conservation based on four different uniform meshes on $\Omega$: (1) $50\times 50$ ($h=0.02$), (2) $66\times 66$ ($h=0.015$), (3) $100\times 100$ ($h=0.01$) and (4) $133\times 133$ ($h=0.0075$). The solid mesh is constructed to have a similar node density to the fluid mesh. In order to visualise the flow a snapshot of the velocity and deformation fields for the first set of parameters is presented in \cref{Snapshot_of_fluid}. For a comparison between the flows energy plots are presented in \cref{energy_evolution}, from which it can be seen that the solid corresponding to the second parameter set is harder and has a larger frequency than the first one. In addition, $E_d(\Omega_t^s)<0$ and both the viscosity and density jump across the interface.

\begin{table}[h!]
\centering
\begin{tabular}{|c|c|c|c|c|c|}
\hline
	Parameter sets  & $\rho^f$ & $\rho^s$ & $\nu^f$ & $\nu^s$ & $\mu^s$\\
	\hline
	Parameter 1 & $1$ & $1$ & $0.01$ & $0.01$ & $1$\\
	\hline
	Parameter 2 & $1$ & $10$ & $0.01$ & $0$ & $10$\\
	\hline
\end{tabular}
\captionsetup{justification=centering}
\caption{Parameter sets for test problem of the oscillating disc.}
\label{Parameter sets for test problem of the oscillating disc}
\end{table}

\begin{figure}[h!]
	\begin{minipage}[t]{0.5\linewidth}
		\centering  
		\includegraphics[width=2in,angle=0]{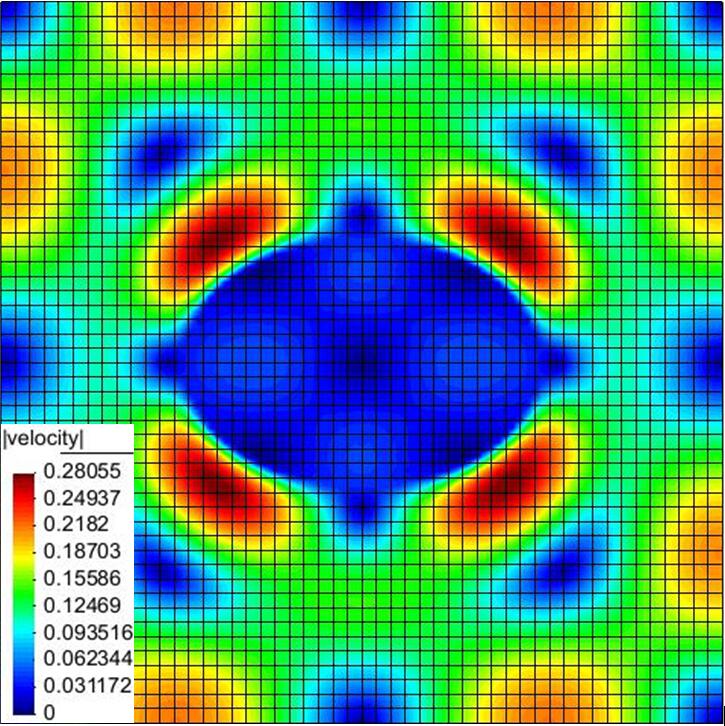}	
		\captionsetup{justification=centering}
		\caption*{\scriptsize(a) Velocity norm on the fluid mesh,}
	\end{minipage}
	\begin{minipage}[t]{0.5\linewidth}
		\centering  
		\includegraphics[width=2.5in,angle=0]{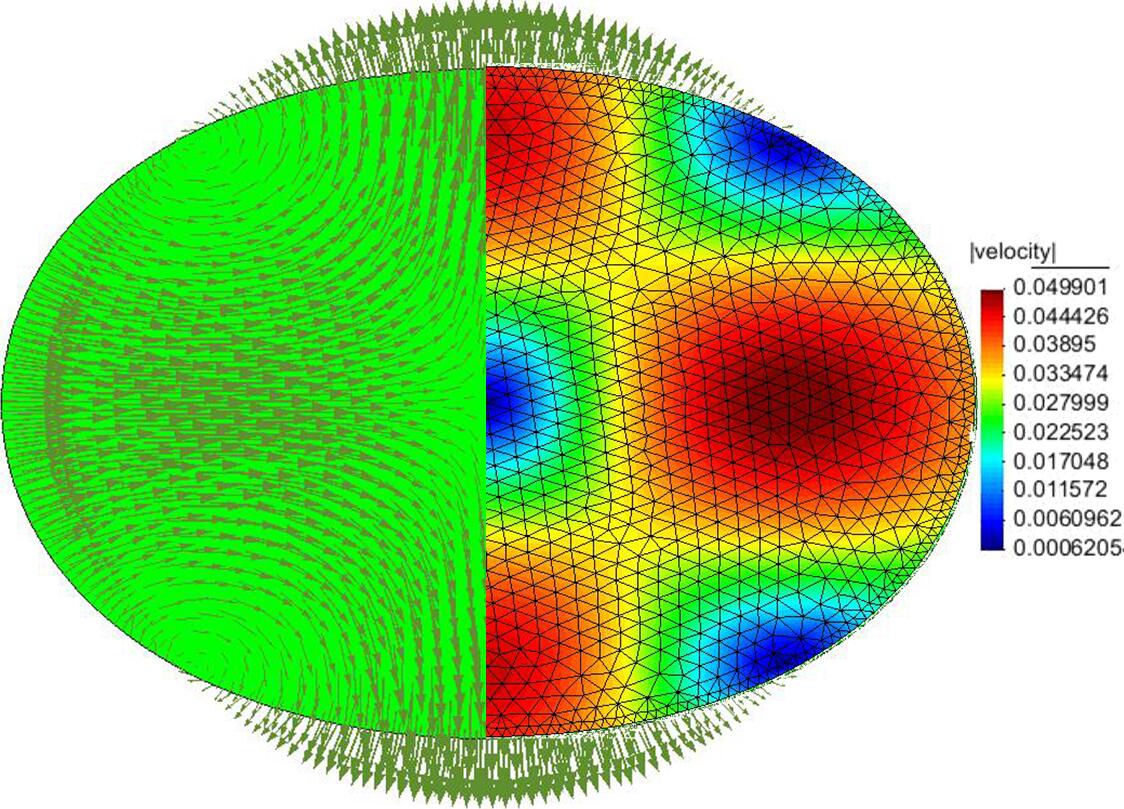}
		\caption*{\scriptsize(b) Distribution of velocity on the solid mesh.}
	\end{minipage}     		
	\captionsetup{justification=centering}
	\caption {\scriptsize Snapshot at $t=0.24$, Parameter 1, $\Delta t=10^{-2}$, on Mesh (1).} 
	\label{Snapshot_of_fluid}
\end{figure}

\begin{figure}[h!]
	\begin{minipage}[t]{0.5\linewidth}
		\centering  
		\includegraphics[width=2.5in,angle=0]{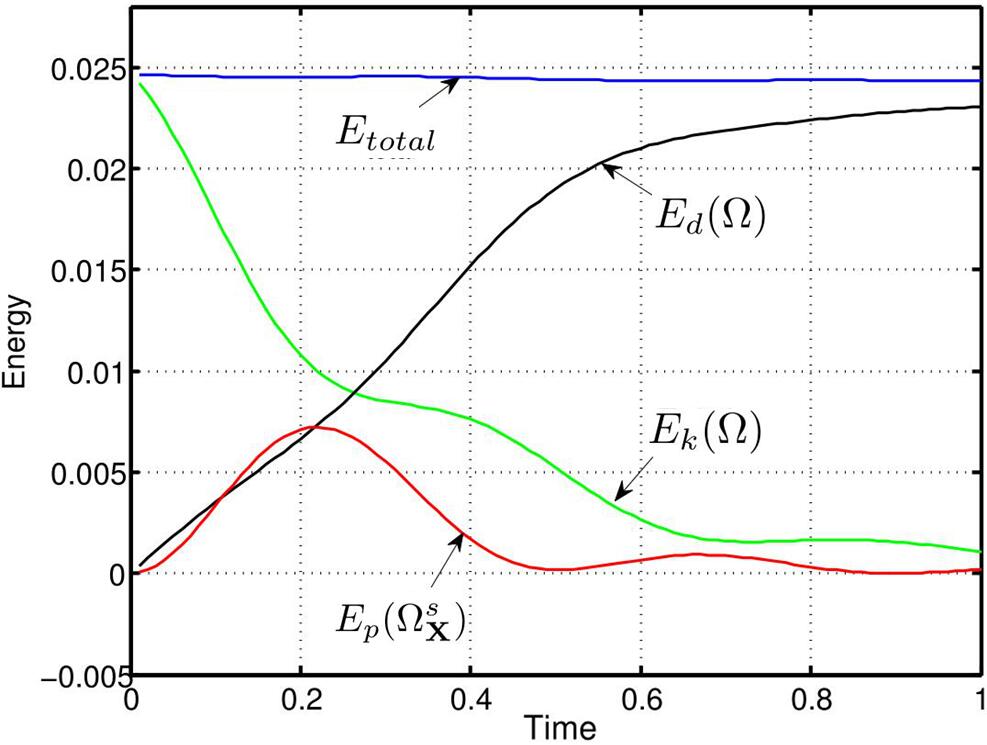}	
		\captionsetup{justification=centering}
		\caption*{\scriptsize(a) Parameter 1 ($E_k(\Omega_t^s)=E_d(\Omega_t^s)=0$),}
	\end{minipage}
	\begin{minipage}[t]{0.5\linewidth}
		\centering  
		\includegraphics[width=2.5in,angle=0]{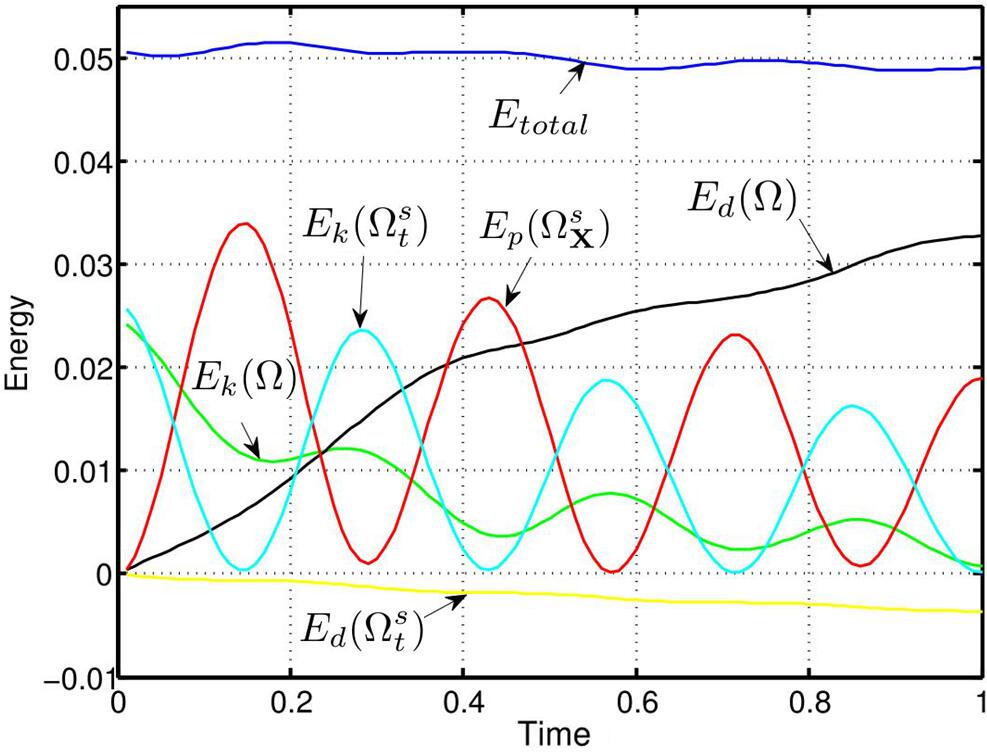}
		\caption*{\scriptsize(b) Parameter 2.}
	\end{minipage}     		
	\captionsetup{justification=centering}
	\caption {\scriptsize Evolution of energy, $\Delta t=10^{-2}$, on Mesh (1).} 
	\label{energy_evolution}
\end{figure}

We commence by comparing the Crank-Nicolson (CN) scheme and backward Euler (BE) scheme, and comparing $P_2/P_1$ elements and $P_2/(P_1+P_0)$ elements. The evolution of mass variation and energy error are demonstrated in \cref{mass_conservation} and \cref{energy_conservation} respectively, from which it is clearly apparent that the Crank-Nicolson scheme has an advantage for both the mass and energy conservation. It can be seen from \cref{mass_conservation} that the enrichment of the pressure field by a constant $P_0$ dramatically improves the mass conservation, although the effect for energy conservation appears to be very slightly negative from \cref{energy_conservation}.

\begin{figure}[h!]
	\begin{minipage}[t]{0.5\linewidth}
		\centering  
		\includegraphics[width=2.5in,angle=0]{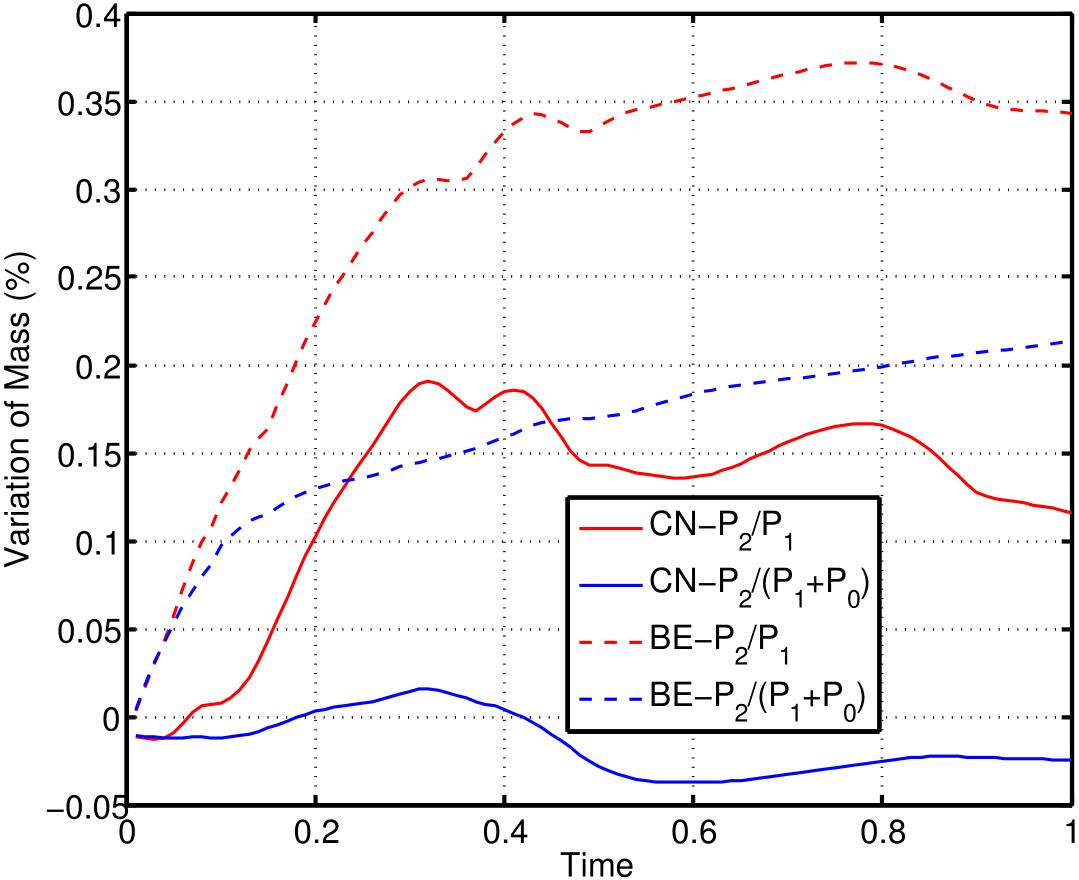}	
		\captionsetup{justification=centering}
		\caption*{\scriptsize(a) Parameter 1,}
	\end{minipage}
	\begin{minipage}[t]{0.5\linewidth}
		\centering  
		\includegraphics[width=2.5in,angle=0]{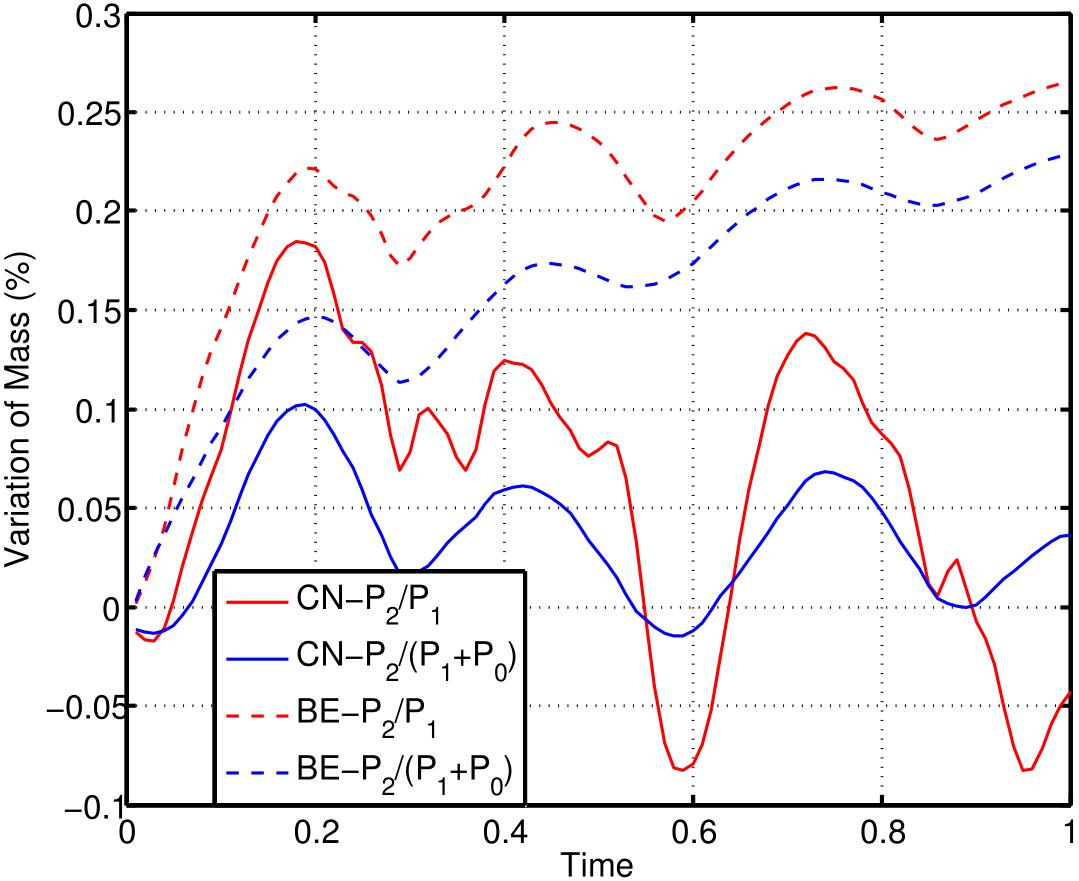}
		\caption*{\scriptsize(b) Parameter 2.}
	\end{minipage}     		
	\captionsetup{justification=centering}
	\caption {\scriptsize Variation of mass against time, $\Delta t=10^{-2}$, on Mesh (1).} 
	\label{mass_conservation}
\end{figure}

\begin{figure}[h!]
	\begin{minipage}[t]{0.5\linewidth}
		\centering  
		\includegraphics[width=2.5in,angle=0]{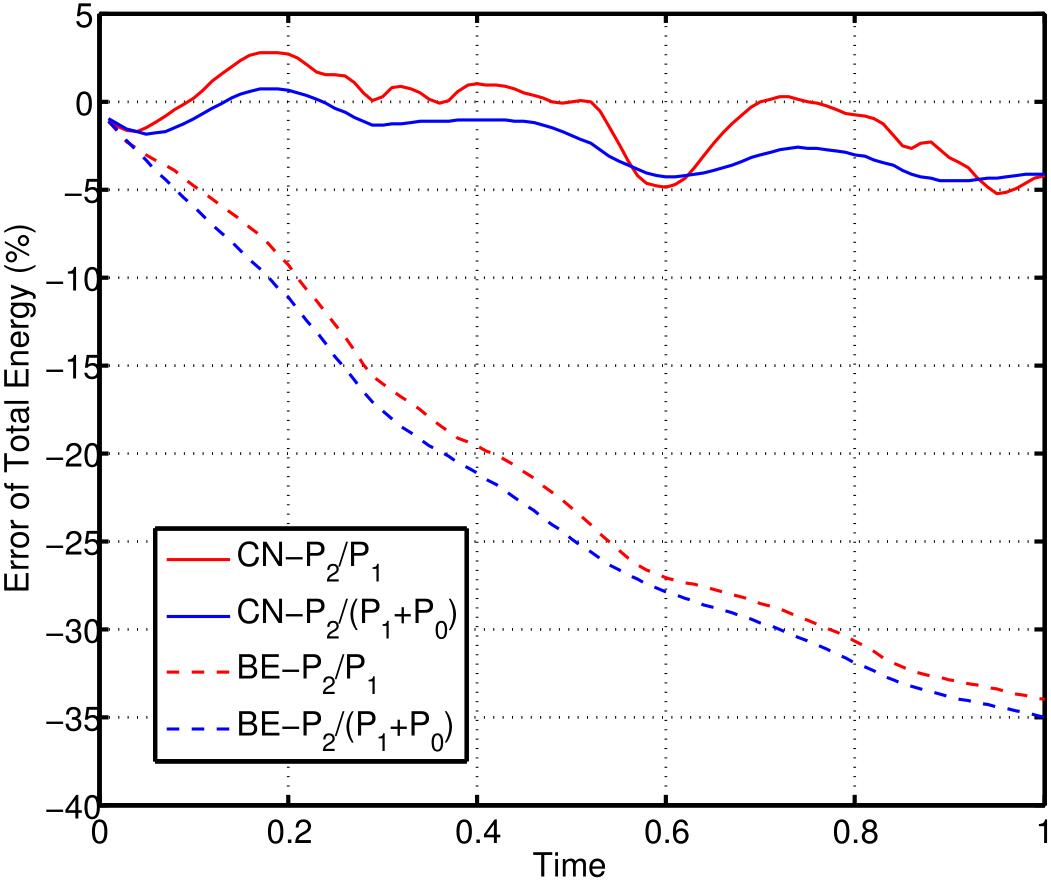}	
		\captionsetup{justification=centering}
		\caption*{\scriptsize(a) Parameter 1,}
	\end{minipage}
	\begin{minipage}[t]{0.5\linewidth}
		\centering  
		\includegraphics[width=2.4in,angle=0]{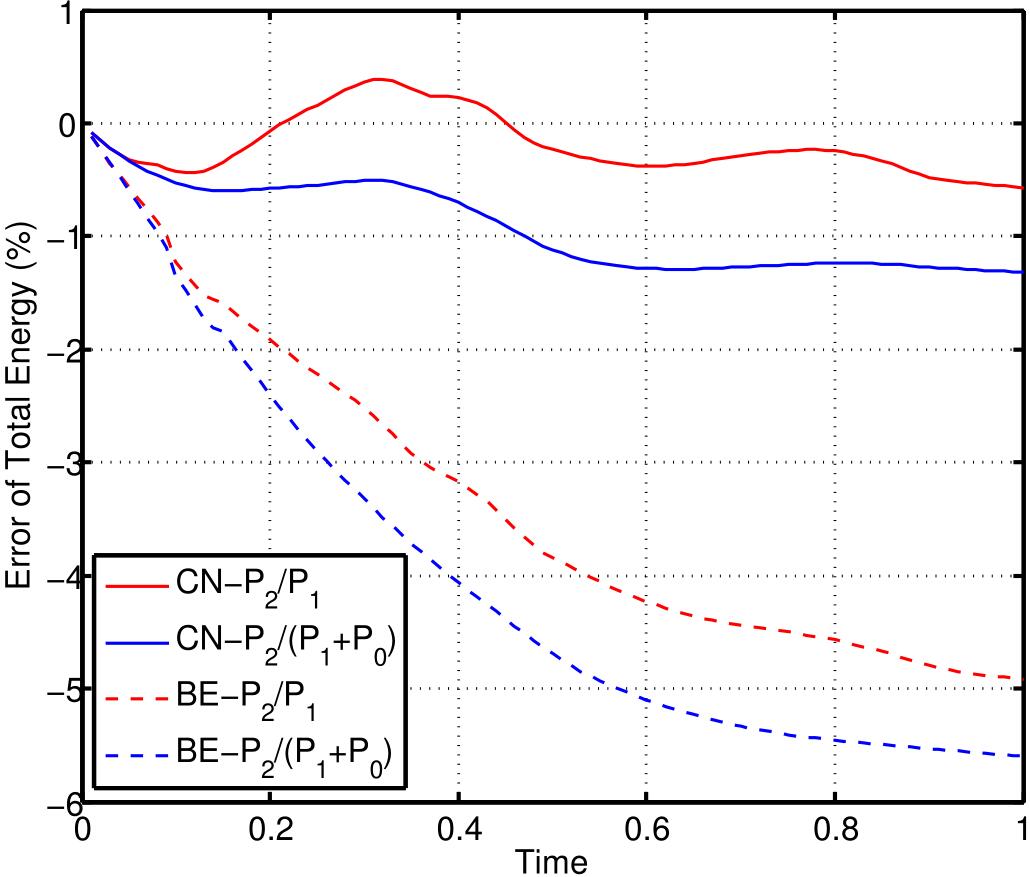}
		\caption*{\scriptsize(b) Parameter 2.}
	\end{minipage}     		
	\captionsetup{justification=centering}
	\caption {\scriptsize Error of total energy ($Err$ defined in \cref{energy estimate_after_time_discretization_closed}) against time, $\Delta t=10^{-2}$, on Mesh (1).} 
	\label{energy_conservation}
\end{figure}

We then use element $P_2/(P_1+P_0)$ and investigate time convergence of the proposed method. It can be observed, from \cref{time_conservation} (a) and \cref{order_of_time_conservation} (a), that both Crank-Nicolson and backward Euler scheme have a first order time convergence (see \cref{reduced_order1} and \cref{reduced_order2} for the reason that Crank-Nicolson scheme is also first order), however the former is more accurate than the latter using the same time step. It can also be observed, from \cref{time_conservation} and \cref{order_of_time_conservation}, that Crank-Nicolson scheme introduces more oscillation than backward Euler scheme in order to gain this accuracy, however, the oscillation can be reduced by mesh refinement (see \cref{mesh_conservation_mass} and \cref{mesh_conservation_energy}).

\begin{figure}[h!]
	\begin{minipage}[t]{0.5\linewidth}
		\centering  
		\includegraphics[width=2.4in,angle=0]{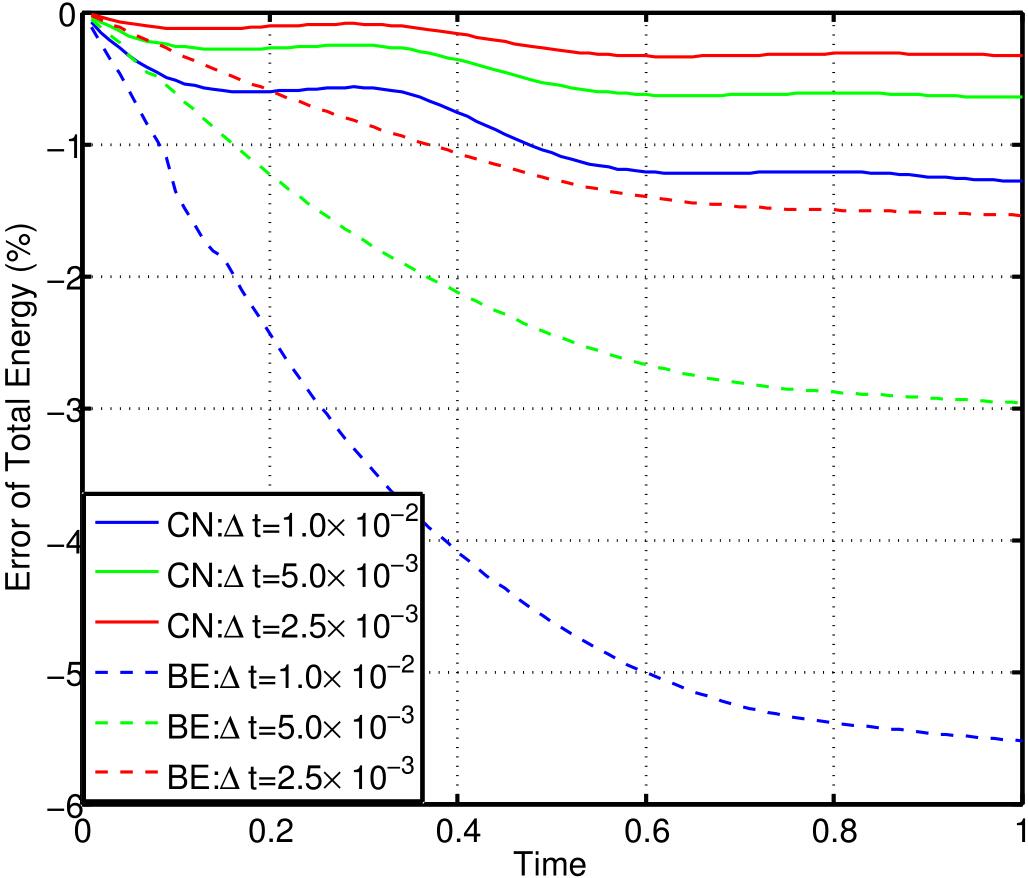}	
		\captionsetup{justification=centering}
		\caption*{\scriptsize(a) Parameter 1,}
	\end{minipage}
	\begin{minipage}[t]{0.5\linewidth}
		\centering  
		\includegraphics[width=2.5in,angle=0]{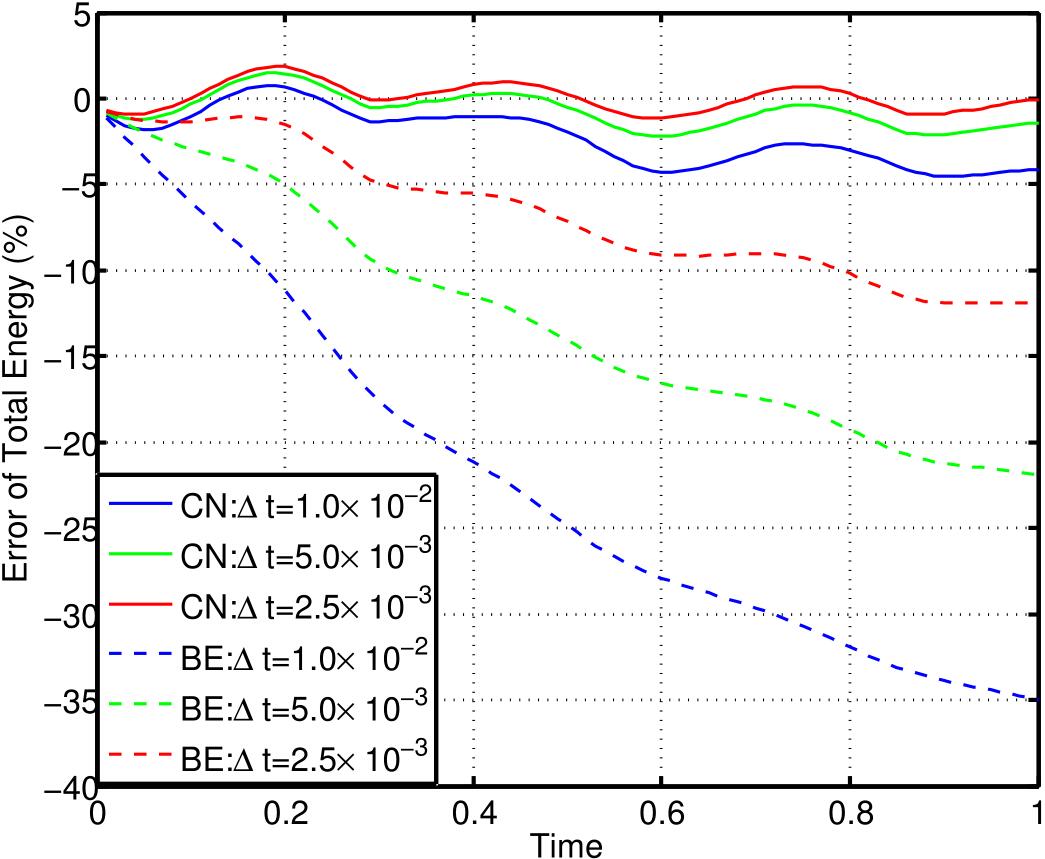}
		\caption*{\scriptsize(b) Parameter 2.}
	\end{minipage}     		
	\captionsetup{justification=centering}
	\caption {\scriptsize Error of total energy ($Err$ defined in \cref{energy estimate_after_time_discretization_closed}) against time, on Mesh (1).} 
	\label{time_conservation}
\end{figure}

\begin{figure}[h!]
	\begin{minipage}[t]{0.5\linewidth}
		\centering  
		\includegraphics[width=2.5in,angle=0]{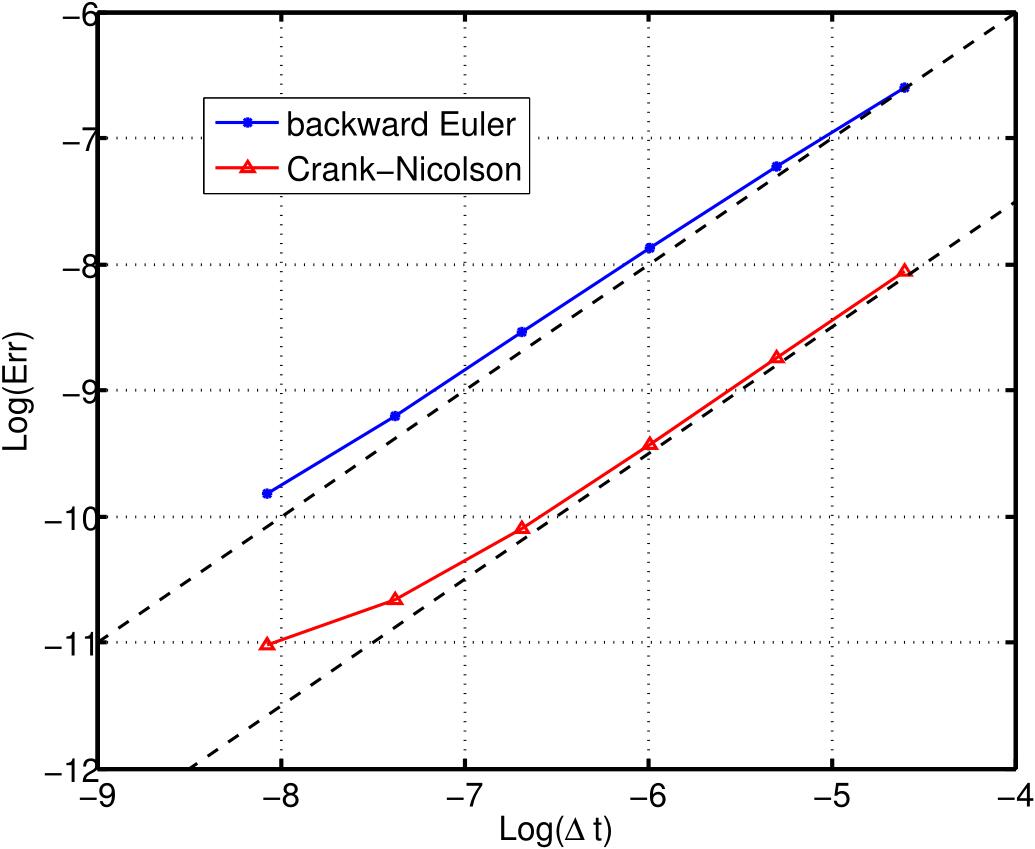}	
		\captionsetup{justification=centering}
		\caption*{\scriptsize(a) Parameter 1,}
	\end{minipage}
	\begin{minipage}[t]{0.5\linewidth}
		\centering  
		\includegraphics[width=2.4in,angle=0]{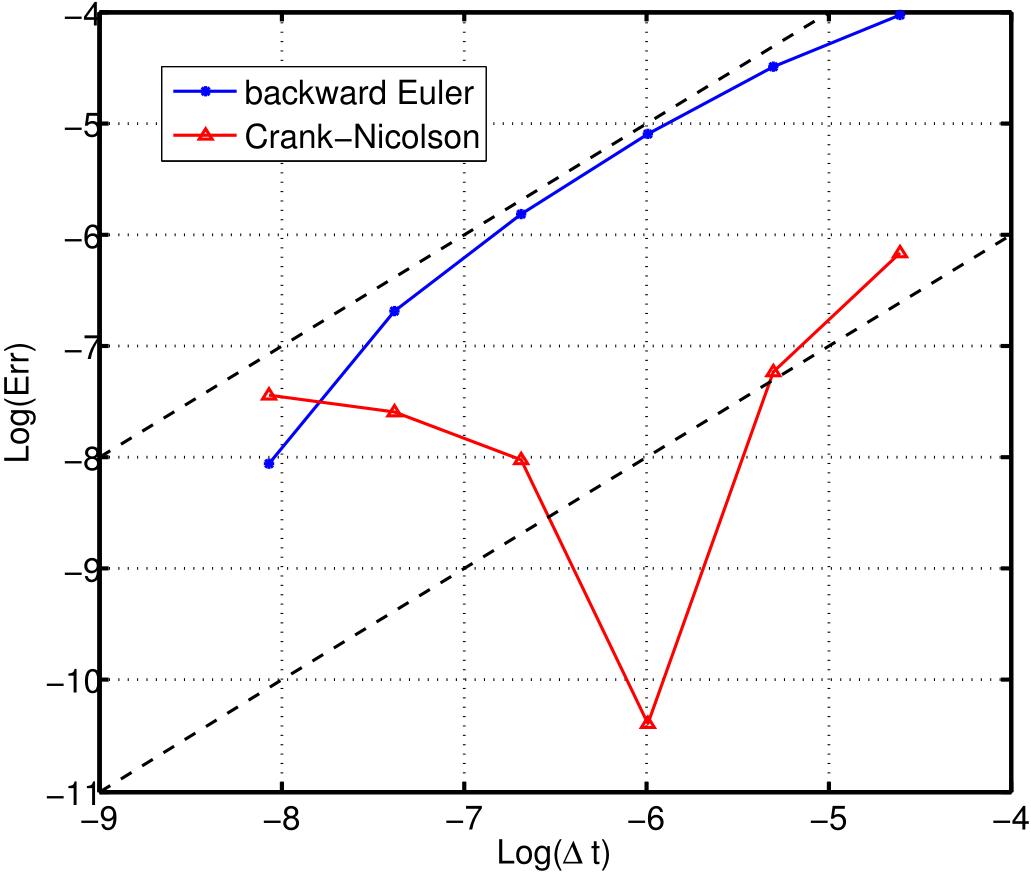}
		\caption*{\scriptsize(b) Parameter 2.}
	\end{minipage}     		
	\captionsetup{justification=centering}
	\caption {\scriptsize Convergence rate of $Err$ defined in \cref{energy estimate_after_time_discretization_closed} at $t=1$, on Mesh (1).} 
	\label{order_of_time_conservation}
\end{figure}

Finally, the mesh convergence is clearly demonstrated in \cref{mesh_conservation_mass} and \cref{mesh_conservation_energy} via mass and energy evolution respectively.

\begin{figure}[h!]
	\begin{minipage}[t]{0.5\linewidth}
		\centering  
		\includegraphics[width=2.5in,angle=0]{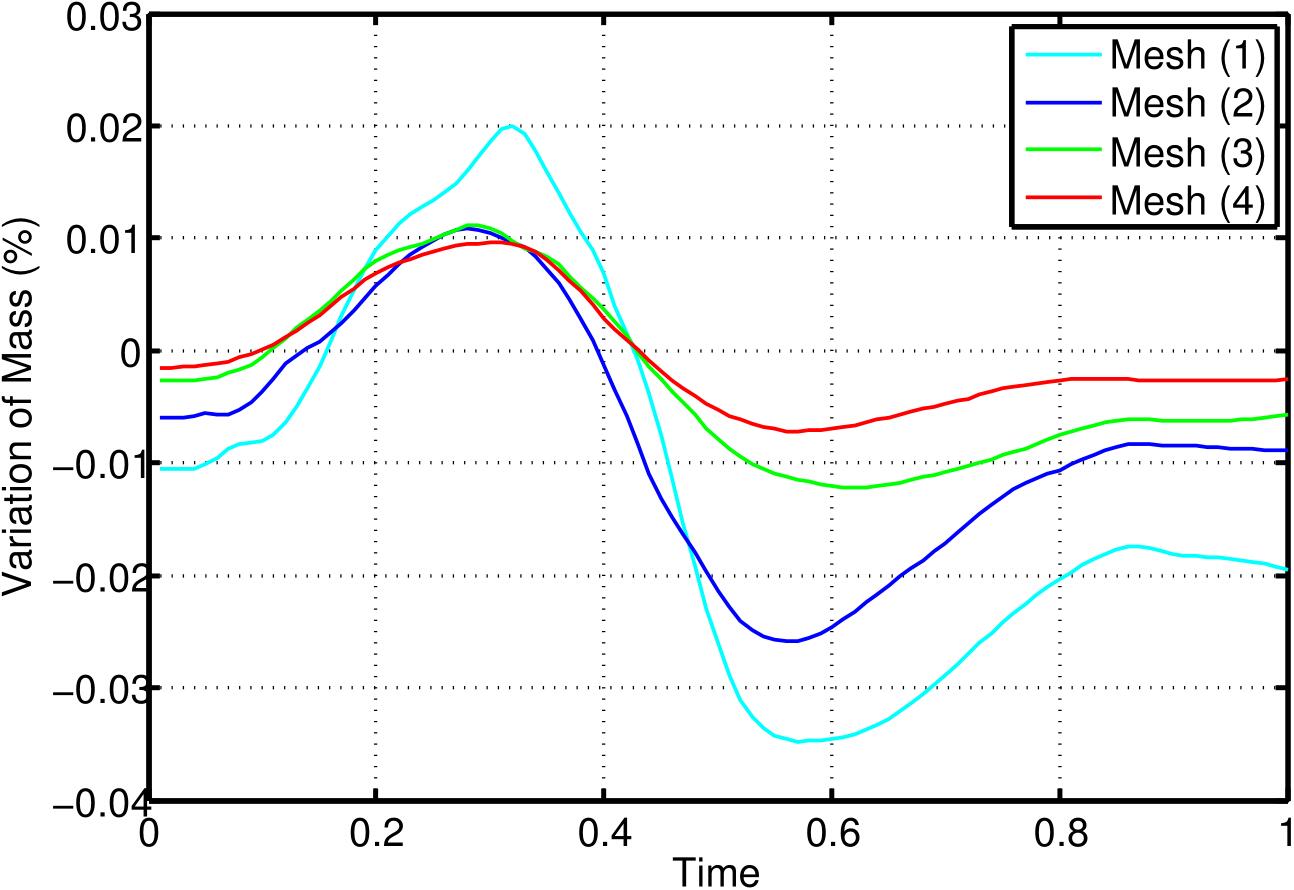}	
		\captionsetup{justification=centering}
		\caption*{\scriptsize(a) Parameter 1,}
	\end{minipage}
	\begin{minipage}[t]{0.5\linewidth}
		\centering  
		\includegraphics[width=2.5in,angle=0]{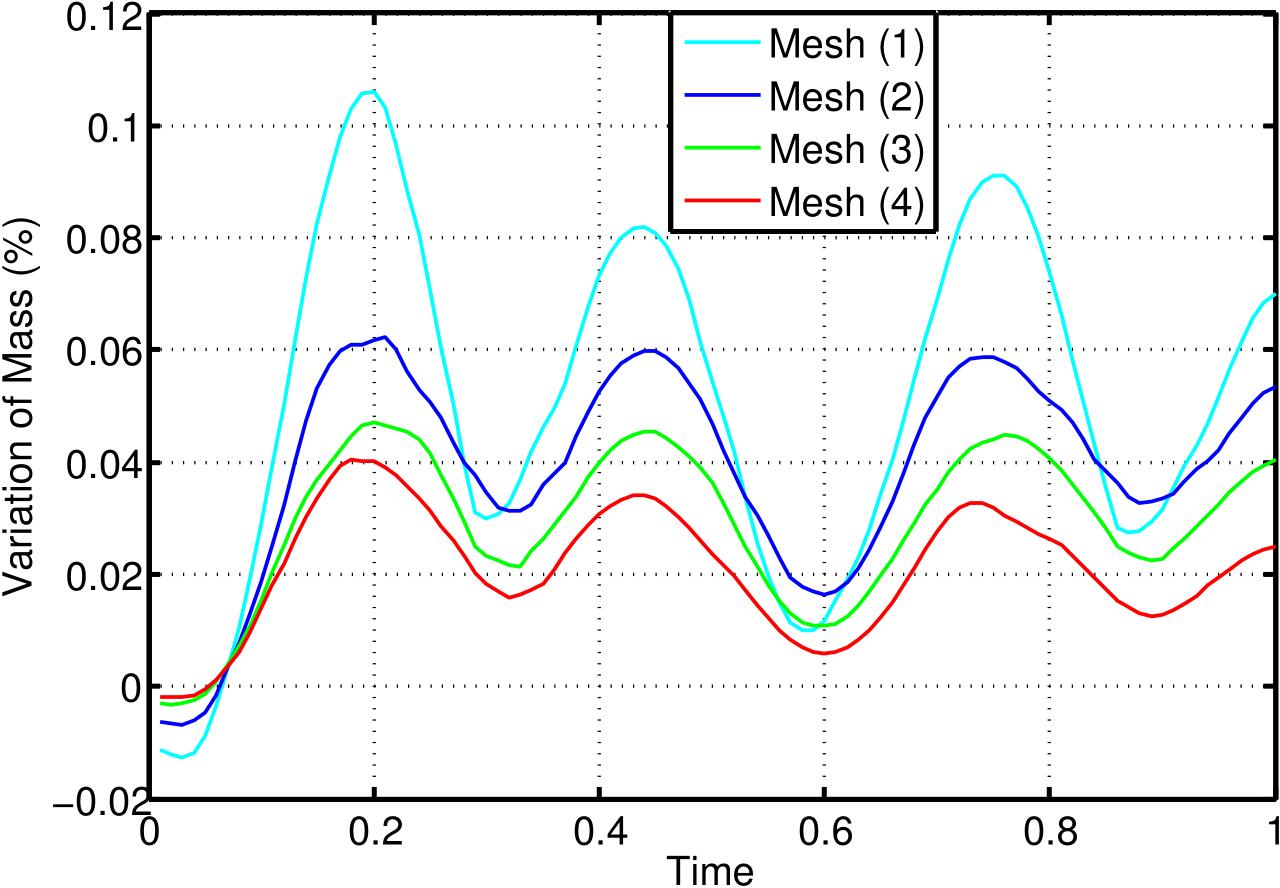}
		\caption*{\scriptsize(b) Parameter 2.}
	\end{minipage}     		
	\captionsetup{justification=centering}
	\caption {\scriptsize Variation of mass against time, $\Delta t=6.25\times 10^{-4}$.} 
	\label{mesh_conservation_mass}
\end{figure}

\begin{figure}[h!]
	\begin{minipage}[t]{0.5\linewidth}
		\centering  
		\includegraphics[width=2.5in,angle=0]{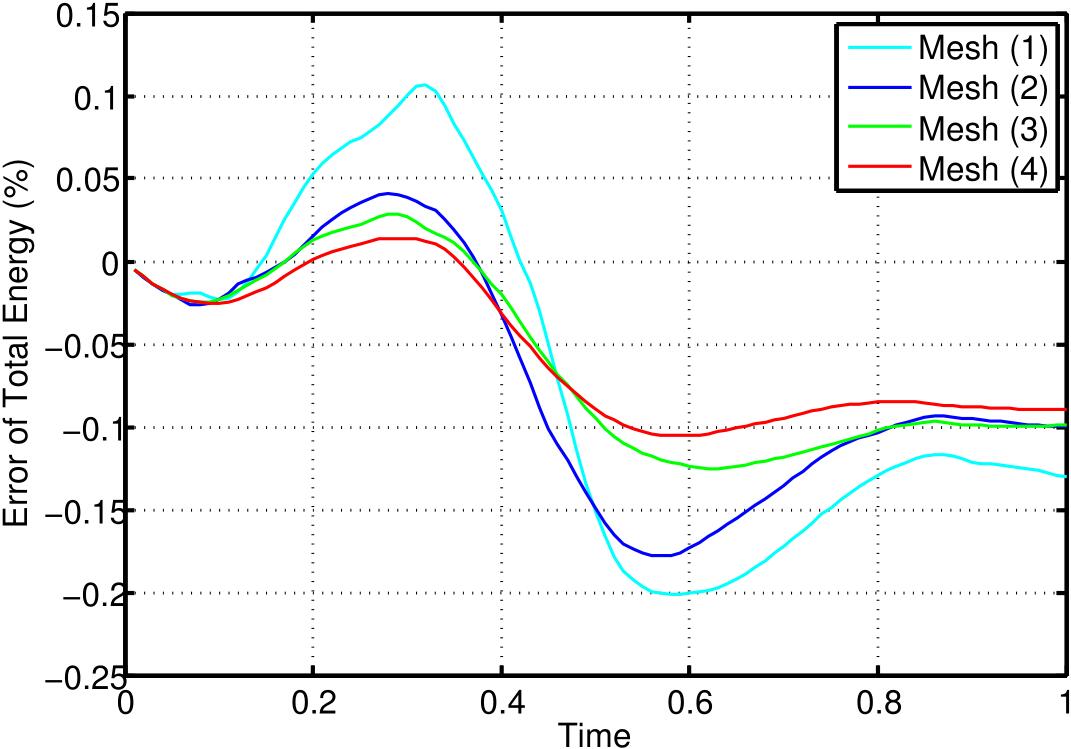}	
		\captionsetup{justification=centering}
		\caption*{\scriptsize(a) Parameter 1,}
	\end{minipage}
	\begin{minipage}[t]{0.5\linewidth}
		\centering  
		\includegraphics[width=2.5in,angle=0]{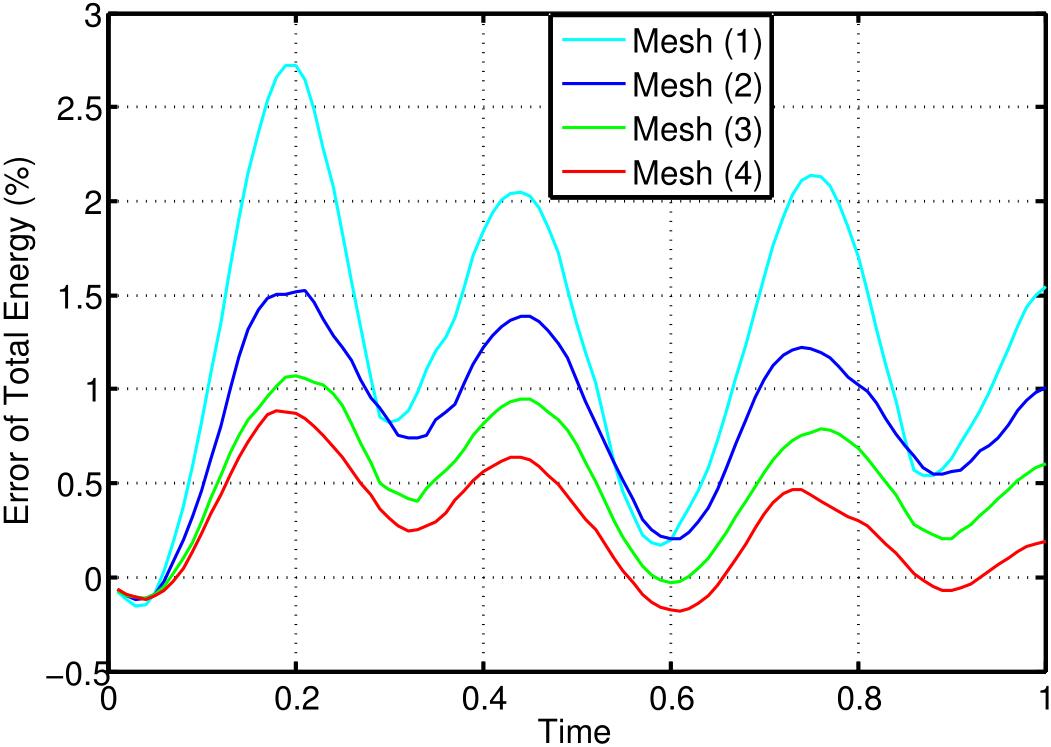}
		\caption*{\scriptsize(b) Parameter 2.}
	\end{minipage}     		
	\captionsetup{justification=centering}
	\caption {\scriptsize Error of total energy against time ($Err$ defined in \cref{energy estimate_after_time_discretization_closed}), $\Delta t=6.25\times 10^{-4}$.} 
	\label{mesh_conservation_energy}
\end{figure}

\subsection{Oscillating disc driven by an initial potential energy}
\label{subsec:odbype}
In the previous example, the disc oscillates because a kinetic energy (the 2nd and 3rd terms in \cref{energy estimate_after_time_discretization_closed}) is prescribed for the FSI system at the beginning. In this test, we shall stretch the disc and create a potential energy in the solid (the last term in \cref{energy estimate_after_time_discretization_closed}), then release it causing the disc to oscillate due to this potential solid energy. The computational domain is a square $\Omega=[0,1]\times[0,1]$. One quarter of a solid disc is located in the left-bottom corner of the square, and initially stretched as an ellipse as shown in \cref{figex2}. Notice the equation of an ellipse $\frac{x^2}{a^2}+\frac{y^2}{b^2}=1$ and its area $\pi ab$, hence we ensure that this stretch does not change mass of the solid.
\begin{figure}[h!]
	\centering
	\includegraphics[width=3in,angle=0]{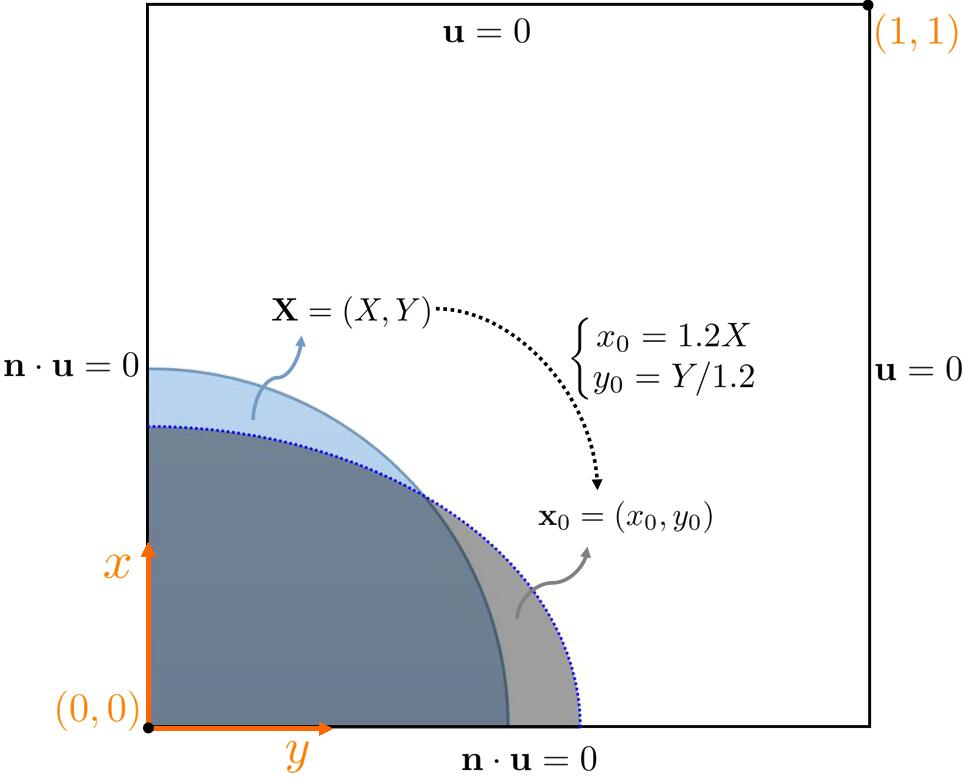}
	\caption {\scriptsize Computational domain and boundary conditions for example \ref{subsec:odbype}.} 
	\label{figex2}
\end{figure}

We choose $\rho^f=1$, $\rho^s=2$, $\mu^s=2$ and $\nu^f=\nu^s=0.01$. The fluid adopts the same meshes defined in section \ref{subsec:oddbaike}, and the solid has similar node density as the fluid. A snapshot of pressure on the fluid mesh and corresponding solid deformation with its velocity norm are displayed in \cref{Pressure and Velocity}. Time and mesh convergence are shown in \cref{time_energy_convergece} and \cref{mesh_energy_convergece} respectively.

\begin{figure}[h!] 
	\begin{minipage}[t]{0.5\linewidth}
	\centering  
	\includegraphics[width=2.5in,angle=0]{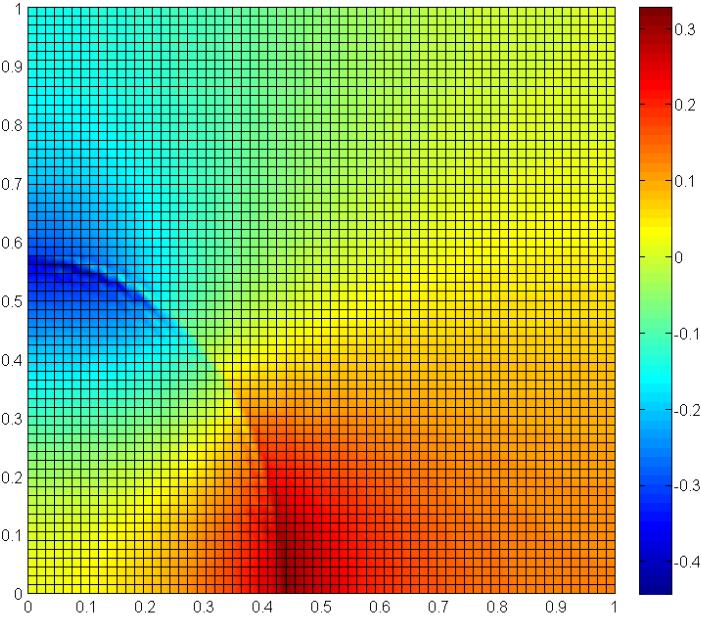}
	\caption*{\scriptsize(a) Distribution of pressures on the fluid mesh,}
	\end{minipage}
	\begin{minipage}[t]{0.5\linewidth}
	\centering  
	\includegraphics[width=1.5in,angle=0]{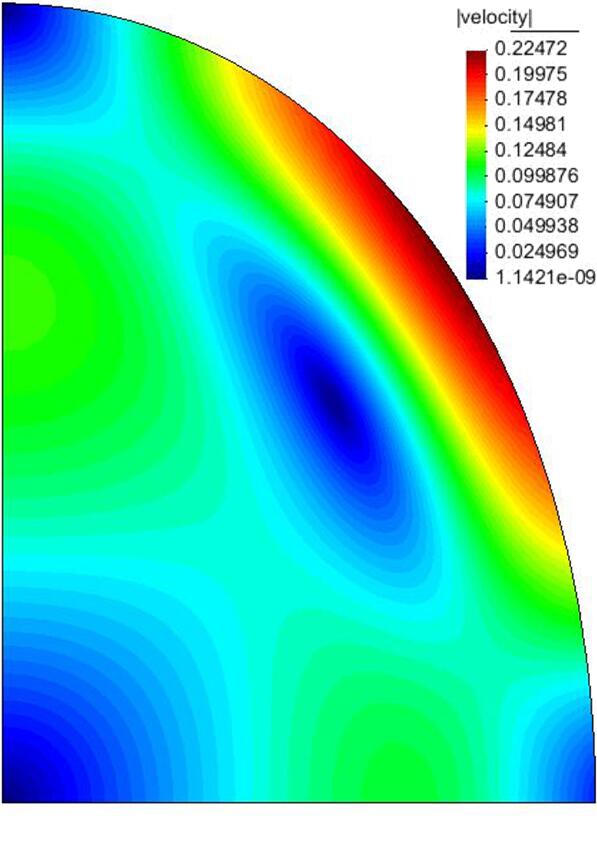}
	\caption*{\scriptsize(b) Velocity norm on the solid.}
	\end{minipage}   		
	\captionsetup{justification=centering}
	\caption {\scriptsize A snapshot at $t=0.9$, $\Delta t=10^{-2}$, on Mesh (2).} 
	\label{Pressure and Velocity}
\end{figure}

\begin{figure}[h!]
	\begin{minipage}[t]{0.5\linewidth}
	\centering  
		\includegraphics[width=2.5in,angle=0]{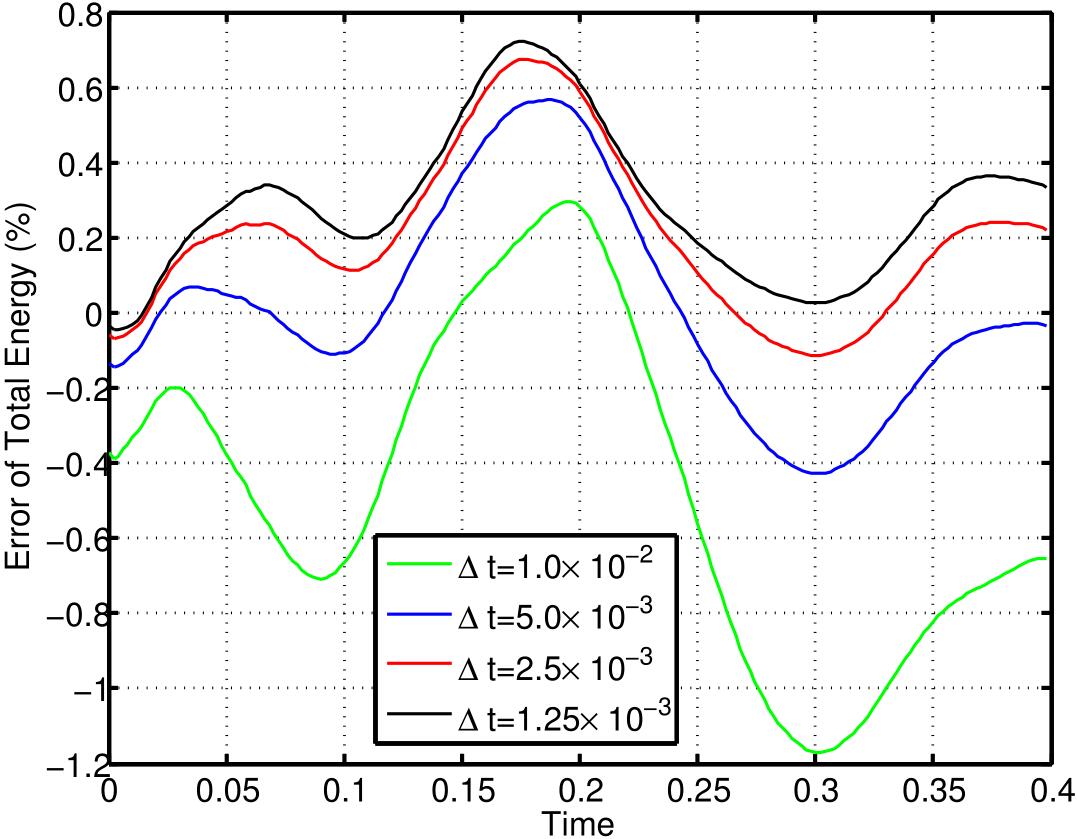}	
	\captionsetup{justification=centering}
	\caption{\scriptsize Time convergence, Mesh (2).}
	\label{time_energy_convergece}		
	\end{minipage}
	\begin{minipage}[t]{0.5\linewidth}
	\centering  
	\includegraphics[width=2.5in,angle=0]{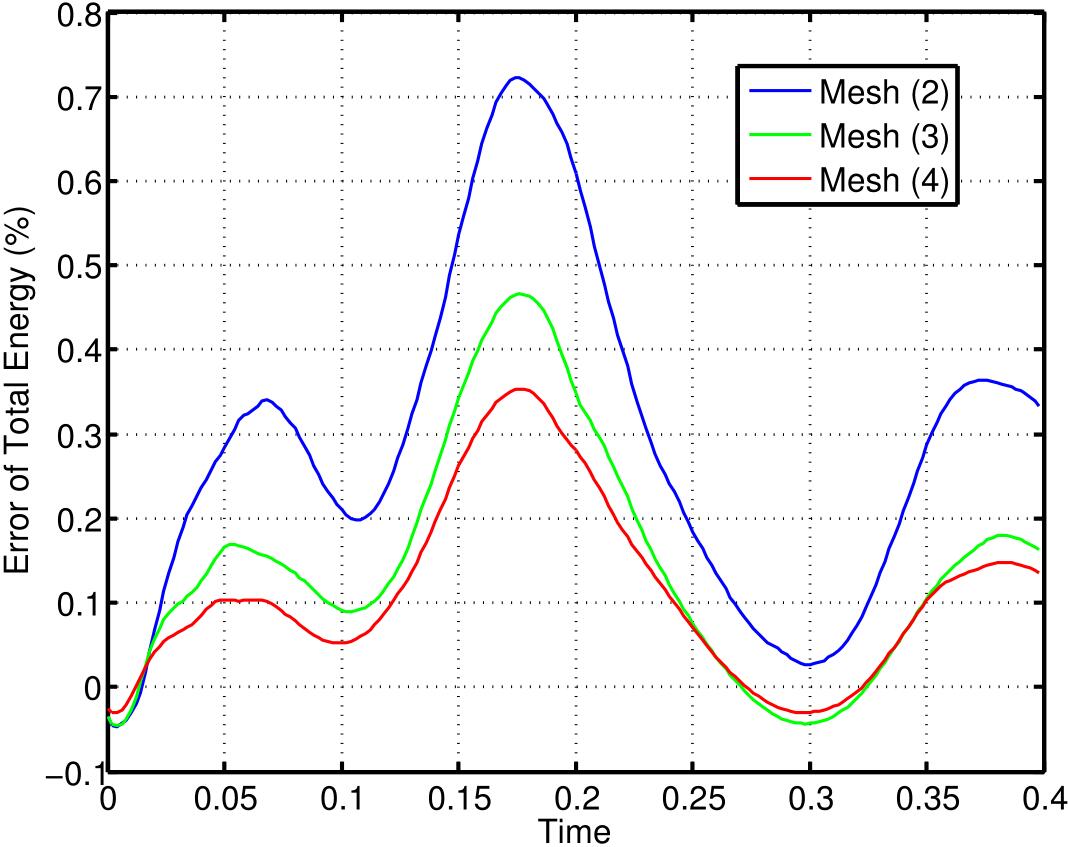}
	\caption{\scriptsize Mesh convergence, $\Delta t=1.25\times 10^{-3}$.}
	\label{mesh_energy_convergece}	
	\end{minipage}     		
\end{figure}

\subsection{Oscillating ball driven by an initial kinetic energy}
\label{subsec:ball}
In this section, we consider a 3D oscillating ball, which is an extension of the example in section \ref{subsec:oddbaike}. The ball is initially located at the center of $\Omega=[0,1]\times[0,1]\times[0,0.6]$ with a radius of $0.2$. Using the property of symmetry this computation is carried out on $1/8$ of domain $\Omega$: $[0,0.5]\times[0,0.5]\times[0,0.3]$. The initial velocities of $x$ and $y$ components are the same as that used in section \ref{subsec:oddbaike} and the $z$ component is set to be 0 at the beginning. We adopt the Parameter set 1 and Mesh (1) defined in section \ref{subsec:oddbaike} (with the same mesh size in z direction). A snapshot of the $1/8$ solid ball and the corresponding fluid velocity norm are presented in \cref{Velocity}, and the result of time convergence is presented in \cref{Time_convergence_ball} and \cref{Order_time_convergence_ball}, from which a first order convergence can be observed.

\begin{figure}[h!]
	\begin{minipage}[t]{0.5\linewidth}
		\centering  
		\includegraphics[width=2.5in,angle=0]{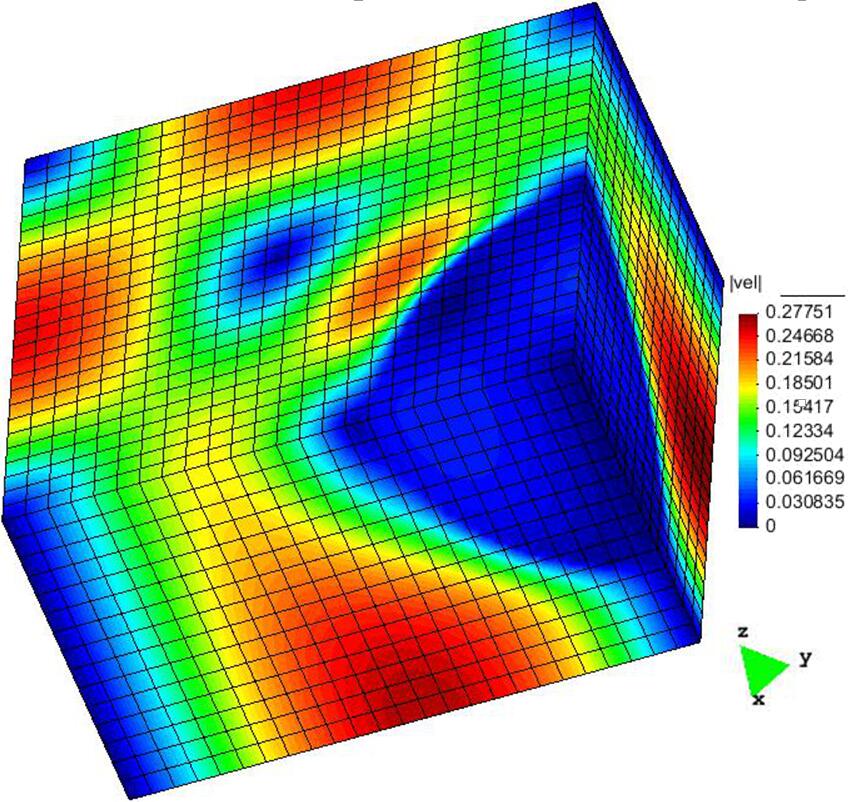}	
		\caption*{\scriptsize(a) On fluid mesh,}
	\end{minipage}
	\begin{minipage}[t]{0.5\linewidth}
		\centering  
		\includegraphics[width=2.0in,angle=0]{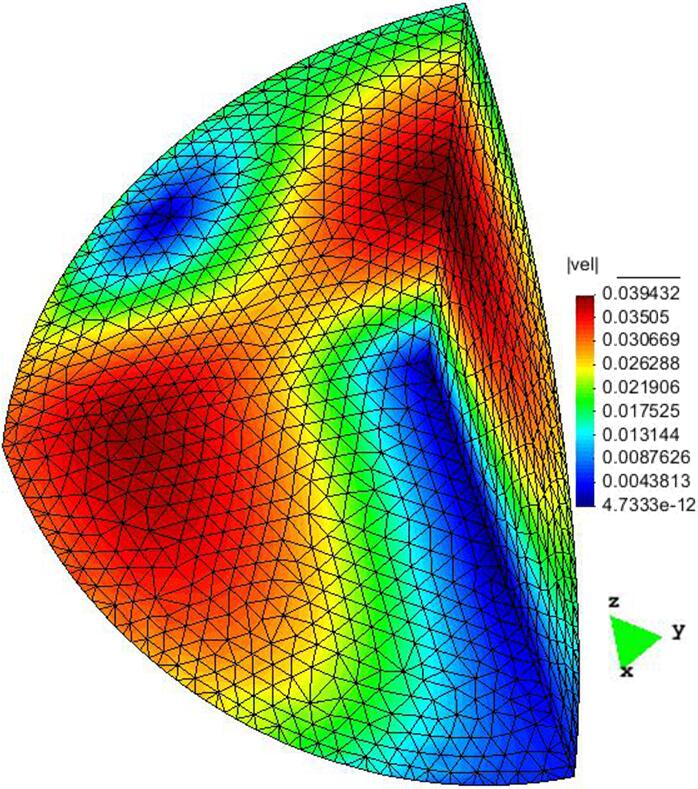}
		\caption*{\scriptsize(b) On solid mesh.}
	\end{minipage}   		
	\captionsetup{justification=centering}
	\caption {\scriptsize Velocity norm at $t=0.2$.} 
	\label{Velocity}
\end{figure}

\begin{figure}[h!]
	\begin{minipage}[t]{0.5\linewidth}
	\centering  
	\includegraphics[width=2.5in,angle=0]{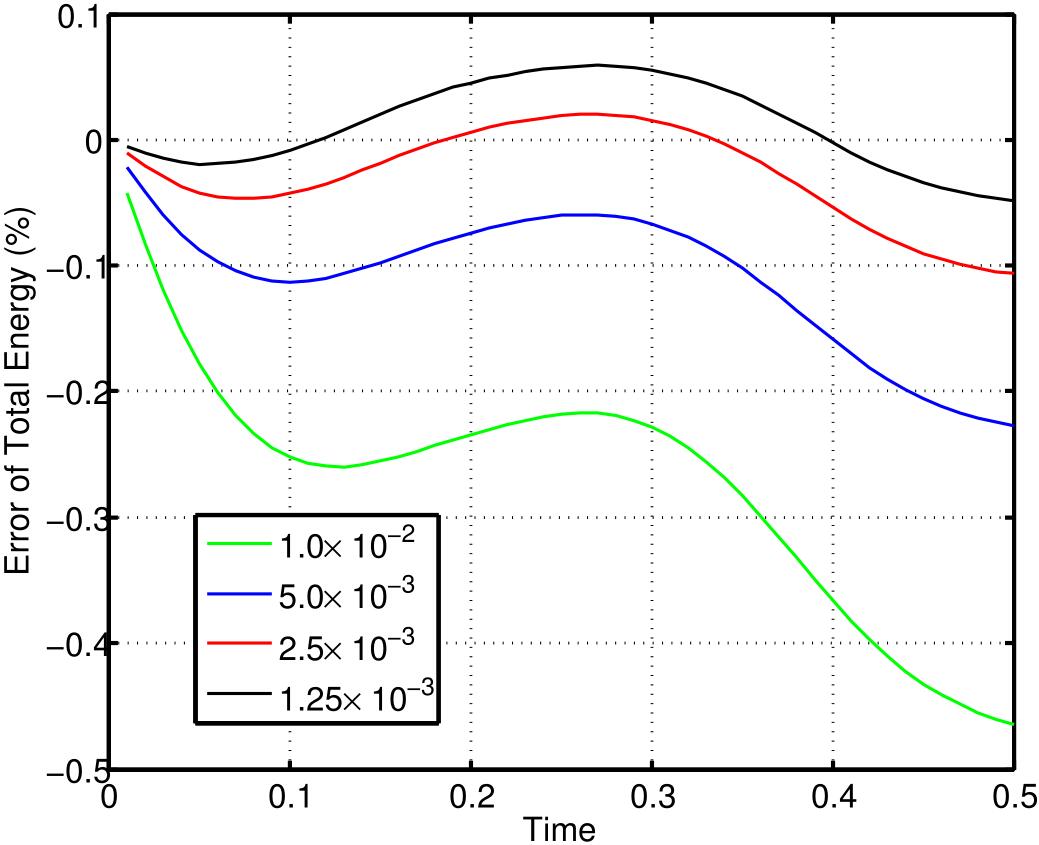}	
	\captionsetup{justification=centering}
	\caption{\scriptsize Error of total energy against time.}
	\label{Time_convergence_ball}		
	\end{minipage}
	\begin{minipage}[t]{0.5\linewidth}
	\centering  
	\includegraphics[width=2.5in,angle=0]{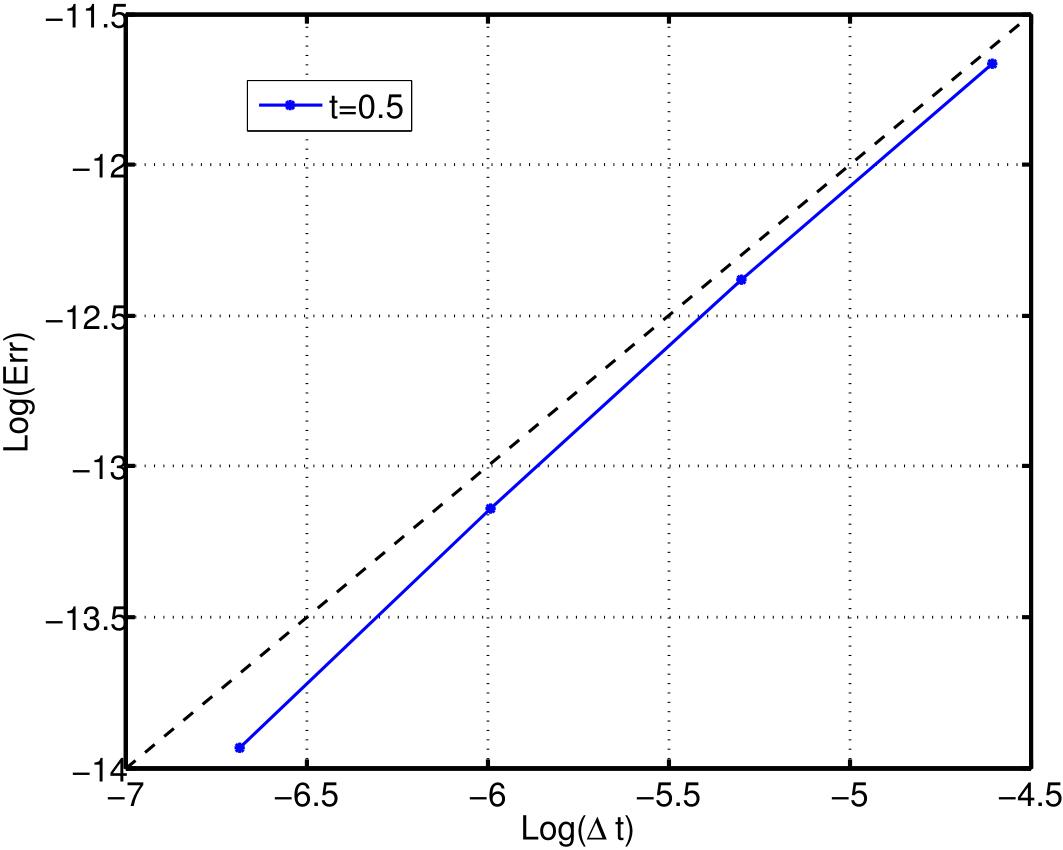}
	\caption{\scriptsize Order of time convergence.}
	\label{Order_time_convergence_ball}			
	\end{minipage}     		
\end{figure}

\section{Conclusions}
\label{sec:conclusions}

In this article, the energy conservation of a new one-field fictitious domain method for fluid-structure interactions is proved after discretization in time. Furthermore, for a special case of the Stokes fluid equation and the neo-Hookean solid model, the well-posedness of the proposed scheme is analyzed and demonstrated.

We then present a selection of numerical tests that demonstrate the theoretical energy estimate in both $2$ and $3$ dimensions. Moreover, we show that the Crank-Nicolson scheme is more accurate than the backward Euler scheme in terms of mass and energy conservation, although both exhibit first order convergence in time (see \cref{reduced_order1} and \cref{reduced_order2} for the reason that Crank-Nicolson scheme is also first order).

\appendix
\section{Energy estimate of the backward Euler scheme} 
\label{appendix:backward-Euler}

\subsection{Weak form}
Using the backward Euler method, the discretized weak form corresponding to \cref{problem:continuous_weak_form} is:
\begin{problem}\label{problem_weak_after_time_discretization_backward_Euler}
For each time step, find ${\bf u}_{n+1}\in H_0^1(\Omega)^d$ and $p_{n+1} \in L_0^2(\Omega)$, such that
\begin{equation}\label{weak_form1_discretization_backward_Euler}
\begin{split}
&\rho^f\int_{\Omega}\frac{{\bf u}_{n+1}-{\bf u}_n}{\Delta t} \cdot{\bf v}d{\bf x}
+\rho^fc\left({\bf u}_{n+1},{\bf u}_{n+1},{\bf v}\right)
+\frac{\nu^f}{2}\int_{\Omega}{\rm D}{\bf u}_{n+1}:{\rm D}{\bf v}d{\bf x} \\
&-\int_{\Omega}p_{n+1}\nabla \cdot {\bf v}d{\bf x}
+\rho^{\delta}\int_{\Omega_{n+1}^s}\frac{{\bf u}_{n+1}-{\bf u}_n}{\Delta t} \cdot{\bf v}d{\bf x}
+\frac{\nu^{\delta}}{2}\int_{\Omega_{n+1}^s}{\rm D}{\bf u}_{n+1}:{\rm D}{\bf v}d{\bf x} \\
&+\int_{\Omega_{\bf X}^s}{\bf P}\left({\bf F}_{n+1}\right):\nabla_{\bf X}{\bf v}d{\bf X}
=O\left(\Delta t\right),
\end{split}
\end{equation}
$\forall{\bf v}\in H_0^1\left(\Omega\right)^d$ and 
\begin{equation}\label{weak_form2_discretization_backward_Euler}
-\int_{\Omega} q\nabla \cdot {\bf u}_{n+1}d{\bf x}=0,
\end{equation}	
$\forall q\in L^2(\Omega)$.
\end{problem}

\begin{remark}\label{update_of_solid_mesh_backward_Euler}
For backward Euler scheme, $\Omega_{n+1}^s$ is updated from $\Omega_n^s$ by
$\Omega_{n+1}^s=\left\{{\bf x}_{n+1}:{\bf x}_{n+1}={\bf x}_n+{\bf u}_{n+1}\Delta t\right\}$, for all ${\bf x}_n\in \Omega_n^s$. 
\end{remark}

\subsection{Energy conservation}
\begin{lemma}\label{energy_estimate1_lemma_backward_Euler}
Assume the solid energy function $\Psi(\cdot)\in C^1$ over the set of second order tensors, then
\begin{equation}\label{energy_estimate1_backward_Euler}
\Delta t{\bf P}\left({\bf F}_{n+1}\right):\nabla_{\bf X}{\bf u}_{n+1}
+O\left(\Delta t^2\right)
=\Psi({\bf F}_{n+1})-\Psi({\bf F}_n).
\end{equation}
\end{lemma}
\begin{proof}
Let $w(\xi)
=\Psi\left({\bf F}_n+\xi\nabla_{\bf X}{\bf u}_{n+1}\right)$, and notice that
\begin{equation*}
\Delta t\nabla_{\bf X}{\bf u}_{n+1}
=\nabla_{\bf X}{\bf x}_{n+1}-\nabla_{\bf X}{\bf x}_n
={\bf F}_{n+1}-{\bf F}_n,
\end{equation*}
then,
\begin{equation*}
\begin{split}
&\Psi({\bf F}_{n+1})-\Psi({\bf F}_n) 
=w\left(\Delta t\right)-w(0)
= \Delta t w^\prime\left(\Delta t\right)
+O\left(\Delta t^2\right).\\
\end{split}
\end{equation*}
Using the chain rule, \cref{energy_estimate1_backward_Euler} holds thanks to
\begin{equation*}
w^\prime\left(\Delta t\right)
=\left.\frac{\partial\Psi}{\partial{\bf F}}\right|_{\xi=\Delta t}:\nabla_{\bf X}{\bf u}_{n+1}
={\bf P}\left({\bf F}_{n+1}\right):\nabla_{\bf X}{\bf u}_{n+1}.
\end{equation*}
\end{proof}

Similarly to \cref{lemma_convection_zero}, we have:
\begin{lemma}\label{lemma_convection_zero_discretization_in_time_euler}
If $\left({\bf u}_{n+1}, p_{n+1}\right)$ is the solution pair of \cref{problem_weak_after_time_discretization_backward_Euler}, then
\begin{equation}\label{convection_zero_discretization_in_time_euler}
\int_{\Omega}\left({\bf u}_{n+1}\cdot\nabla\right){\bf u}_{n+1}\cdot{\bf u}_{n+1}d{\bf x}=0.
\end{equation}
\end{lemma}

\begin{proposition} [Local Energy Conservation]\label{lec_backward_Euler}
Let $\left({\bf u}_{n+1}, p_{n+1}\right)$ be the solution pair of \cref{problem_weak_after_time_discretization_backward_Euler}, then
\begin{equation}\label{energy estimate_after_time_discretization_backward_Euler}
\begin{split}
&\frac{\rho^f}{2}\int_{\Omega}\left|{\bf u}_{n+1}\right|^2d{\bf x}
+\frac{\rho^\delta}{2}\int_{\Omega_{n+1}^s}\left|{\bf u}_{n+1}\right|^2d{\bf x}
+\int_{\Omega_{\bf X}^s}\Psi\left({\bf F}_{n+1}\right)d{\bf X}  \\
&+\frac{\Delta t\nu^f}{2}\int_{\Omega}{\rm D}{\bf u}_{n+1}:{\rm D}{\bf u}_{n+1}d{\bf x}
+\frac{\Delta t\nu^\delta}{2}\int_{\Omega_{n+1}^s}{\rm D}{\bf u}_{n+1}:{\rm D}{\bf u}_{n+1}d{\bf x} \\
&\le \frac{\rho^f}{2}\int_{\Omega}\left|{\bf u}_n\right|^2d{\bf x}
+\frac{\rho^\delta}{2}\int_{\Omega_n^s}\left|{\bf u}_n\right|^2d{\bf x}
+\int_{\Omega_{\bf X}^s}\Psi\left({\bf F}_n\right)d{\bf X}+O\left(\Delta t^2\right).
\end{split}
\end{equation}
\end{proposition}
\begin{proof}
Let ${\bf v}={\bf u}_{n+1}$ in \cref{weak_form1_discretization_backward_Euler}, $q=p_{n+1}$ in \cref{weak_form2_discretization_backward_Euler} and use \cref{lemma_convection_zero_discretization_in_time_euler}, we have:
\begin{equation}\label{one_equation_backward_Euler}
\begin{split}
&\rho^f\int_{\Omega}\left({\bf u}_{n+1}-{\bf u}_n\right)\cdot{\bf u}_{n+1}d{\bf x}
+\frac{\Delta t\nu^f}{2}\int_{\Omega}{\rm D}{\bf u}_{n+1}:{\rm D}{\bf u}_{n+1}d{\bf x}\\
&+\rho^{\delta}\int_{\Omega_{n+1}^s}\left({\bf u}_{n+1}-{\bf u}_n\right)\cdot{\bf u}_{n+1}d{\bf x}
+\frac{\Delta t\nu^{\delta}}{2}\int_{\Omega_{n+1}^s}{\rm D}{\bf u}_{n+1}:{\rm D}{\bf u}_{n+1}d{\bf x}\\
&+\Delta t\int_{\Omega_{\bf X}^s}{\bf P}\left({\bf F}_{n+1}\right):\nabla_{\bf X}{\bf u}_{n+1}d{\bf X}
=O\left(\Delta t^2\right).
\end{split}
\end{equation}
Using Cauchy-Schwarz inequality and the fact $ab\le\frac{a^2+b^2}{2}$, we have:
\begin{equation*}
\int_{\omega}{\bf u}_n\cdot{\bf u}_{n+1}d{\bf x}
\le \left\|{\bf u}_n\right\|_{0,\omega} \left\|{\bf u}_{n+1}\right\|_{0,\omega}
\le \frac{\|{\bf u}_n\|^2+\|{\bf u}_{n+1}\|^2}{2},
\end{equation*}
where $\omega=\Omega$ or $\Omega_{n+1}^s$. Substituting the above equation into \cref{one_equation_backward_Euler}, we get \cref{lec_backward_Euler} due to \cref{energy_estimate1_lemma_backward_Euler}.
\end{proof}

\begin{corollary} [Total Energy Conservation]\label{cor:tec_backward_Euler}
Let $N+1=t/\Delta t$, where $t$ is the computational time, and denote the error of total energy as:
\begin{equation}\label{energy estimate_after_time_discretization_closed_backward_Euler}
\begin{split}
&Err=\frac{\rho^f}{2}\int_{\Omega}\left|{\bf u}_{N+1}\right|^2d{\bf x}
+\frac{\rho^\delta}{2}\int_{\Omega_{N+1}^s}\left|{\bf u}_{N+1}\right|^2d{\bf x} 
+\int_{\Omega_{\bf X}^s}\Psi\left({\bf F}_{N+1}\right)d{\bf X} \\
&+\frac{\Delta t\nu^f}{2}\sum_{n=0}^{N}\int_{\Omega}{\rm D}{\bf u}_{n+1}:{\rm D}{\bf u}_{n+1}d{\bf x}
+\frac{\Delta t\nu^\delta}{2}\sum_{n=0}^{N}\int_{\Omega_{n}^s}{\rm D}{\bf u}_{n+1}:{\rm D}{\bf u}_{n+1}d{\bf x} \\
&- \frac{\rho^f}{2}\int_{\Omega}\left|{\bf u}_0\right|^2d{\bf x}
-\frac{\rho^\delta}{2}\int_{\Omega_0^s}\left|{\bf u}_0\right|^2d{\bf x}
-\int_{\Omega_{\bf X}^s}\Psi\left(\frac{\partial{\bf x}_0}{\partial {\bf X}}\right)d{\bf X},
\end{split}
\end{equation}
then,
\begin{equation*}
Err\le O(\Delta t).
\end{equation*}
\end{corollary}
\begin{proof}
First let $O(\Delta t^2)=\sum_{k\ge 2}C_{n+1}^k\Delta t^k$ in \cref{energy estimate_after_time_discretization_backward_Euler}, where $C_{n+1}^k$ is a constant dependent on the specific time step. Then add equation \cref{energy estimate_after_time_discretization_backward_Euler} from $n=0$ to $n=N$, we have $Err\le \sum_{k\ge 2}\Delta t^k\sum_{n=0}^{N}C_{n+1}^k$. \cref{energy estimate_after_time_discretization_closed_backward_Euler} holds due to $\sum_{n=0}^{N}C_{n+1}^k \le C_{max}^k t/\Delta t$, where $C_{max}^k=max\{C_{n+1}^k, n=0,\cdots,N\}$.
\end{proof}


\end{document}